\documentclass[11 pt,colorinlistoftodos]{article}

\usepackage{amsmath,amsthm,amsfonts,amssymb,csquotes,epsfig, titlesec, epsfig, float, caption,wrapfig, mathrsfs}
\usepackage[hidelinks]{hyperref}
\usepackage{tikz-cd}
\usepackage{tikz}
\usepackage{tikz-3dplot}
\usepackage{xcolor}
\usepackage{verbatim}
\usepackage{accents}
\usepackage{marvosym}
\usepackage{multirow}
\usepackage{multicol}
\usepackage[citestyle=alphabetic,bibstyle=alphabetic,backend=bibtex, maxbibnames=10,minbibnames=5]{biblatex}
\usepackage[colorinlistoftodos]{todonotes}
\usepackage[]{enumerate}
\usepackage[american]{babel}

\usetikzlibrary{calc, decorations.pathreplacing,shapes.misc}
\usetikzlibrary{decorations.pathmorphing}
\usepackage[left=1in,top=1in,right=1in]{geometry}
\usepackage[capitalize]{cleveref}
\usepackage{subcaption}
\usetikzlibrary{shapes.geometric}

\newtheorem{thm}{Theorem}[section]

\newtheorem{lemma}[thm]{Lemma}
\newtheorem{prop}[thm]{Proposition}

\newtheorem{cor}[thm]{Corollary}
\newtheorem{corollary}[thm]{Corollary}
\newtheorem{conj}[thm]{Conjecture}
\newtheorem{conjecture}[thm]{Conjecture}
\newtheorem{question}[thm]{Question}

\theoremstyle{remark} 
\newtheorem{rem}[thm]{Remark}
\newtheorem{remark}[thm]{Remark}
 \crefname{rem}{Remark}{Remarks}
 \Crefname{rem}{Remark}{Remarks}

\newtheorem{example}[thm]{Example}

\theoremstyle{definition} 
 
\newtheorem{definition}[thm]{Definition} 

\titleformat*{\section}{\normalsize \bfseries \filcenter}
\titleformat*{\subsection}{\normalsize \bfseries }
\newtheorem{mainthm}{Theorem}
\Crefname{mainthm}{Theorem}{Theorems}

\Crefname{maincor}{Corollary}{Corollaries}
\newtheorem{mainquestion}[mainthm]{Question}
\Crefname{mainquestion}{Question}{Questions}
\newtheorem{mainconj}[mainthm]{Conjecture}
\Crefname{mainconj}{Conjecture}{Conjectures}

\crefformat{equation}{(#2#1#3)}

\makeatletter
\def\namedlabel#1#2{\begingroup
   \def\@currentlabel{#2}
   \label{#1}\endgroup
}
\makeatother

\DeclareFontFamily{U}{mathx}{}
\DeclareFontShape{U}{mathx}{m}{n}{<-> mathx10}{}
\DeclareSymbolFont{mathx}{U}{mathx}{m}{n}
\DeclareMathAccent{\widehat}{0}{mathx}{"70}
\DeclareMathAccent{\widecheck}{0}{mathx}{"71}

\renewbibmacro{in:}{}

 \usepackage{amsmath}

\title{\normalsize \textbf{Relating categorical dimensions in topology and symplectic geometry}}
\author{\normalsize Andrew Hanlon, Jeff Hicks, Oleg Lazarev}
\date{}

\newcommand{\eps}{\varepsilon}

\newcommand{\RR}{\mathbb R}
\newcommand{\ZZ}{\mathbb Z}
\newcommand{\CC}{\mathbb C}

\newcommand{\NN}{\mathbb N}

\newcommand{\LL}{\mathfrak c}

\newcommand{\CP}{\mathbb{CP}}

\newcommand{\into}{\hookrightarrow}
\newcommand{\tensor}{\otimes}

\renewcommand{\Im}{\text{Im}}

\newcommand{\cocore}{\mathfrak{u}}
\newcommand{\stp}{\mathfrak{f}}

\newcommand{\gentime}{\text{\ClockLogo}}

\renewcommand{\mod}{\textbf{mod}}

\DeclareMathOperator{\cuplength}{Cuplength}
\DeclareMathOperator{\Grot}{Grot}
\DeclareMathOperator{\hocolim}{hocolim}

\DeclareMathOperator{\Ddim}{Ddim}
\DeclareMathOperator{\Res}{Res}
\DeclareMathOperator{\Cat}{Cat}
\DeclareMathOperator{\GCdim}{GCdim}
\DeclareMathOperator{\Perf}{Perf}
\DeclareMathOperator{\critval}{cVal}
\DeclareMathOperator{\actval}{aVal}
\DeclareMathOperator{\id}{id}

\DeclareMathOperator{\cone}{cone}

\DeclareMathOperator{\LS}{LS}

\DeclareMathOperator{\Coh}{Coh}
\DeclareMathOperator{\Crit}{Crit}

\DeclareMathOperator{\st}{\; |\; }

\DeclareMathOperator{\Ob}{Ob}

\DeclareMathOperator{\Rdim}{Rdim}

\newcommand{\Addresses}{{\bigskip
  \footnotesize

  \noindent A.~Hanlon, \textsc{Department of Mathematics, University of Oregon}\par\nopagebreak
  \noindent \textit{E-mail address}: \texttt{ahanlon@uoregon.edu}

  \medskip

  \noindent J.~Hicks, \textsc{School of Mathematics and Statistics,  University of St Andrews}\par\nopagebreak
  \noindent \textit{E-mail address}: \texttt{jeff.hicks@st-andrews.ac.uk}

  \medskip

  \noindent O.~Lazarev, \textsc{Department of Mathematics,  University of Massachusetts Boston}\par\nopagebreak
  \noindent \textit{E-mail address}: \texttt{oleg.lazarev@umb.edu}

  \medskip

}}

\begin{document}

\newpage
\setcounter{page}{1}
\maketitle
\begin{abstract} 
We study several notions of dimension for (pre-)triangulated categories naturally arising from topology and symplectic geometry. We prove new bounds on these dimensions and raise several questions for further investigation. For instance, we relate the Rouquier dimension of the wrapped Fukaya category of either the cotangent bundle of a smooth manifold $M$ or more generally a Weinstein domain $X$ to quantities of geometric interest. These quantities include the minimum number of critical values of a Morse function on $M$, the Lusternik-Schnirelmann category of $M$, the number of distinct action values of a Hamiltonian diffeomorphism of $X$, and the smallest $n$ such that $X$ admits a Weinstein embedding into $\mathbb{R}^{2n+1}$. Along the way, we introduce a notion of the Lusternik-Schnirelmann category for dg-categories and construct exact Lagrangian cobordisms for restriction to a Liouville subdomain. 
\end{abstract}
\section{Introduction}
Given a triangulated category $\mathcal{C}$ (with a dg or $A_\infty$ enhancement), it is natural to define integer-valued invariants measuring the complexity of $\mathcal{C}$. Such ``dimensions" might capture how difficult it is to build $\mathcal{C}$ from generating objects or simpler subcategories. For example, the Rouquier dimension $\Rdim(\mathcal{C})$ \cite{rouquier2008dimensions} measures the complexity of $\mathcal{C}$ through the number of cones needed to relate objects. Roughly speaking, $\Rdim(\mathcal{C})$ is the minimum over generators of $\mathcal{C}$ of the maximum number of cones needed to build any object of $\mathcal{C}$ from the given generator up to other natural categorical operations. In another direction, the diagonal dimension $\Ddim(\mathcal{C})$ \cite{ballard2012hochschild} is defined to be the minimal generation time of the diagonal bimodule over $\mathcal{C}$ by product objects. See \cite{elagin2021three} for further discussion of these and other notions of dimension for triangulated categories. Under a variety of hypotheses (e.g., $\mathcal{C}$ is proper and linear over a field), \cite{rouquier2008dimensions,ballard2012hochschild,elagin2021three,bai2021rouquier} have proven that 
$$ \Rdim(\mathcal{C}) \leq \Ddim(\mathcal{C}).$$
Here, we study these dimensions when $\mathcal{C}$ is (the homotopy category of) an $A_\infty$ category naturally associated with a symplectic manifold. Along the way, we also introduce some new topologically inspired notions of categorical dimension. 

To be more precise, we will focus on the wrapped Fukaya category $\mathcal{W}(X)$ of a Weinstein manifold $X$. A Weinstein manifold is a non-compact symplectic manifold constructed via handle attachments, and is the symplectic analog of a CW complex. The Floer-theoretic invariants of Weinstein manifolds are particularly well-behaved. The wrapped Fukaya category was introduced in \cite{abouzaid2010open} with alternate definitions given and properties further described in \cite{sylvan2019partially,ganatra2017covariantly, ganatra2018sectorial}. \citeauthor{bai2021rouquier} \cite{bai2021rouquier} undertook the first comprehensive study of the Rouquier dimension of wrapped Fukaya categories of Weinstein manifolds using arborealization \cite{nadler2017arboreal,alvarez2020positive} as a key input. 

We take a different approach starting from the fundamental example of a Weinstein manifold given by the cotangent bundle $T^*M$ of a smooth manifold $M$. The pre-triangulated closure of $\mathcal{W}(T^*M)$ is quasi-equivalent to the category of perfect modules $\Perf C_{-*} (\Omega M)$ over chains on the based loop space $\Omega M$, or equivalently, the category of $\infty$-local systems \cite{abouzaid2011cotangent,abouzaid2012wrapped}. 
Thus, one expects a purely topological and homotopy invariant answer to the following question.

\begin{mainquestion} \label{mainq:cotangent}
    What are the Rouquier dimension and diagonal dimension of $\mathcal{W}(T^*M) \simeq \Perf C_{-*} (\Omega M)$ for a smooth manifold $M$?
\end{mainquestion}

In \cite{bai2021rouquier} (see also \cite{favero2023rouquier}), the authors show that over a field
\begin{equation} \label{eq:bccotangent}
    \Rdim(\mathcal{W}(T^*M)) \leq \dim(M) 
\end{equation} 
using sectorial descent \cite{ganatra2018sectorial}. In some instances, such as when $M$ is a torus, \eqref{eq:bccotangent} is sharp, but in general, $\Rdim(\mathcal{W}(T^*M))$ and even $\Ddim(\mathcal{W}(T^*M))$ can be smaller. For example, $\Rdim(\mathcal{W}(T^*S^n)) = \Ddim(\mathcal{W}(T^*S^n)) = 1$ for every $n$-dimensional sphere $S^n$ with $n \geq 1$ when working over a field. 

It is also interesting to pose \cref{mainq:cotangent} for an arbitrary Liouville manifold $X$. In general, \cite{bai2021rouquier} uses arborealization \cite{nadler2017arboreal,alvarez2020positive} and sectorial descent \cite{ganatra2018sectorial} to extend the argument for \cref{eq:bccotangent} to obtain
\begin{equation} \label{eq:bcgeneral}
    \Rdim(\mathcal{W}(X)) \leq \dim(X) - 3
\end{equation}
for any polarizable Weinstein manifold $X$ when $\dim(X)  \geq 6$ and 
\begin{equation} \label{eq:bclowd}
    \Rdim(\mathcal{W}(X)) = \dim(X)/2 
\end{equation} 
when $\dim(X) \leq 6$ (when working over a field). We are not aware of a general geometrically meaningful lower bound for $\Rdim(\mathcal{W}(X))$, but \cite[Corollary 1.17]{bai2021rouquier} provides a lower bound in some circumstances.

We take a more direct route to improve \eqref{eq:bccotangent} in general and \eqref{eq:bcgeneral} in certain cases. We additionally extend some applications of \cite{bai2021rouquier}. 
Along the way, we introduce other topologically inspired integral measures of complexity for a (pre-triangulated) $A_\infty$ category, and we generalize \cite[Proposition 1.37]{ganatra2018sectorial} by removing assumptions on a tool for using exact Lagrangian cobordisms in the wrapped Fukaya category.

\subsection{Results}

We now state our main results more precisely. For simplicity, we will assume that we are working over a field throughout this summary, but in the text we provide versions which hold over arbitrary commutative rings. In this section, we will also be ambiguous about grading/orientation data on the wrapped Fukaya category. As above, we will write $\Ddim(\mathcal{C})$ and $\Rdim(\mathcal{C})$ for the diagonal dimension and Rouquier dimension of the idempotent and pre-triangulated closure of an $A_\infty$ category $\mathcal{C}$ for brevity.
We begin with our best upper bounds for \cref{mainq:cotangent}. 

\begin{mainthm}[\cref{cor:Morsebound}, \cref{prop:LSboundLS}] \label{mainthm:cotangent}
    If $M$ is a smooth manifold and $f$ is any Morse function on $M$,
    \begin{equation} \label{eq:ddimcot}
        \Ddim(\mathcal{W}(T^*M)) \leq |\critval(f)| - 1
    \end{equation} 
    where $\critval(f)$ is the set of critical values of $f$ and $|\critval(f)|$ is the number of critical values, and
    \begin{equation} \label{eq:rdimcot}
        \Rdim(\mathcal{W}(T^*M)) \leq \LS(M)
    \end{equation}
    where $\LS(M)$ is the Lusternik-Schnirelmann category of $M$. \end{mainthm}

In both \eqref{eq:ddimcot} and \eqref{eq:rdimcot}, the upper bound is witnessed by the generator supported on a cotangent fiber. In \cref{prop:lower_bound_cuplength}, we show that the generation time of this object is bounded below by the cuplength of the cohomology of $M$ as was independently proved in \cite[Theorem 2.14]{favero2023rouquier}.
We note that this bound also has connections to classical results bounding below the number of self-intersection points of the zero section under Hamiltonian perturbation in \cref{subsubsec:cuplength}. 

To prove \eqref{eq:ddimcot}, we give an explicit resolution of the diagonal bimodule by constructing a cobordism using the Morse function $f$. A key point is that this twisted complex is filtered by the critical values of $f$. To deduce that this cobordism induces an isomorphism in the wrapped Fukaya category of the product, we prove a generalization of \cite[Proposition 1.37]{ganatra2018sectorial} in \cref{prop:lagrangianCobordism}. To obtain \eqref{eq:rdimcot}, we employ sectorial descent \cite{ganatra2018sectorial} for a covering of $M$ by null-homotopic sets. Although this strategy is similar to that employed in \cite{bai2021rouquier}, it relies on an additional observation that the Rouquier dimension of a homotopy colimit can be bounded by the generation time for the \emph{essential images} of the constituent pieces. This last point leads us to define the triangulated LS category $\LS(\mathcal{C})$ of an $A_\infty$ category $\mathcal{C}$ roughly as the longest path in a diagram of functors such that the Rouquier dimension of the essential image of each component is $0$ and the diagram computes $\mathcal{C}$ as a homotopy colimit. A similar argument also allows us to bound the diagonal dimension by the LS-category of $M$ in \cref{rem:diagonalLS}.

The bounds from \cref{mainthm:cotangent} are specializations of the following bounds which hold for Weinstein domains. 

\begin{mainthm}[\cref{prop:ddimgeneral},  \cref{eq:lssect}] \label{mainthm:Weinstein}
    If $X$ is a Weinstein domain with skeleton $\mathfrak{c}_X$ and $\phi$ is any Hamiltonian diffeomorphism of $X$ induced by a Hamiltonian $H$ such that the intersection points $\mathfrak{c}_X \cap \phi(\mathfrak{c}_X)$ are transverse and contained in the smooth strata, then 
    \begin{equation} \label{eq:ddimgeneral}
        \Ddim(\mathcal{W} (X)) \leq \left| A_H\left(\mathfrak{c}_X \cap \phi(\mathfrak{c}_X) \right) \right| -1,
    \end{equation}
    where $A_H$ is the action of $H$ (see \eqref{eq:sympact}) and
    \begin{equation} \label{eq:rdimgeneral}
        \Rdim(\mathcal{W}(X)) \leq \LS(\mathcal{W}(X)) \leq \LS_{sect}(X)
    \end{equation}
    where 
    \begin{itemize}
        \item  $\left| A_H\left(\mathfrak{c}_X \cap \phi(\mathfrak{c}_X) \right) \right|$ is the cardinality of the image of $\mathfrak{c}_X \cap \phi(\mathfrak{c}_X) $  under the function $A_H$, and
        \item     $\LS_{sect}(X)$ is the minimal depth of a sectorial cover of $X$ by sectors whose wrapped Fukaya categories have essential image with $\Rdim = 0$ in $\mathcal{W}(X)$. 
    \end{itemize}
\end{mainthm}

We similarly have an extension to a certain class of non-transverse intersections in \cref{subsubsec:degenerate}.

In \cref{prop: attaching_sphere_reeb_chords}, we show that the right-hand side of \eqref{eq:ddimgeneral} can be replaced with the action values of the Reeb chords of the attaching Legendrians given a self-indexing Weinstein Morse function on $X$ in a certain class. The bound \eqref{eq:ddimgeneral} is a strengthening of \cite[Proposition 1.14]{bai2021rouquier} where the upper bound is given simply by the number of intersection points. Using \eqref{eq:ddimgeneral} and \eqref{eq:rdimgeneral}, we are also able to improve \cite[Proposition 1.15]{bai2021rouquier} relating categorical dimensions to Lefschetz fibrations on $X$. Note that \cite{giroux2017existence} shows that every Weinstein domain admits a Lefschetz fibration.

\begin{mainthm}[\cref{prop:Lefschetzddim}, \cref{prop: Lefschetz_covering_bound}] \label{mainthm:lef}
    If $\pi \colon X \to \CC$ is a Lefschetz fibration on a Weinstein domain $X$, 
        \[ \Ddim(\mathcal{W}(X)) \leq \critval(\pi)-1 \hspace{1 cm} \text{ and } \hspace{1cm} \LS_{sect}(X) \leq \critval(\pi) -1 \]
    where $\critval(\pi)$ is the number of critical values of $\pi$.
\end{mainthm}

Finally, we are able to improve \eqref{eq:bcgeneral} under stricter hypotheses by using a generalization of \eqref{eq:bccotangent} to stopped cotangent bundles and known embedding results into cotangent bundles. 

\begin{mainthm}[\cref{cor:trivialembed}] \label{mainthm:embed}
    Suppose that $X$ is a Weinstein domain admitting a Lagrangian distribution $E$ on its tangent bundle that is a trivial vector bundle. If $\mathcal{W}(X; E)$ is the wrapped Fukaya category defined with the grading/orientation data determined by $E$, then 
        $$ \Rdim ( \mathcal{W}(X; E)) \leq \frac{\dim(X)}{2} + 1 .$$
\end{mainthm}

\subsection{Further context and questions}

The results of the previous section suggest several avenues for further research and have connections to other interesting problems.

There are classical relations among the upper bounds in \cref{mainthm:cotangent}. For instance, it is well-known \cite{james1978category} that the minimum number of critical points of a smooth function on $M$ is closely related to the Lusternik-Schnirelmann category $\LS(M)$ of $M$.  
In fact, there are upper bounds \cite{james1978category} 
\begin{align*}
\LS(M)\le \min_{f:M\to \RR} |\Crit(f)| - 1  && \LS(M)\leq \min_{f:M\to \RR \text{ Morse}} |\critval(f)|
\end{align*}
For brevity of notation, let $\mathcal C(M) := \Perf C_* (\Omega M)$ be the category of perfect modules over chains on the loop (or equivalently path) space of $M$ over a field. 
\begin{table}
    \centering
    \begin{tikzpicture}[scale=1.5]

\node (v4) at (-1.5,1.5) {$\LS(M)$};
\node (v2) at (-5.5,1.5) {$\LS_{cat}(\mathcal C(M))$};
\node (v5) at (2,1.5) {$\min_{f \text{ Morse}}|\critval{f}|$};
\node (v3) at (2,-1.5) {$\text{Ddim}(\mathcal C(M))$};
\node (v1) at (-5.5,-1.5) {$\text{Rdim}(\mathcal C(M))$};
\node (v6) at (-1.5,2.5) {$\GCdim(M)$};
\node (v7) at (2,2.5) {$\text{dim}(M)$};
\draw  (v1) edge[->] node[fill=white] {\cref{prop:TLSCatToRdim}} (v2);
\draw  (v1) edge[->] node[fill=white, below] {\cref{prop:diagGenerationToGeneration}} (v3);
\draw  (v2) edge[->] node[fill=white, above] {\cref{prop:LSboundLS}}(v4);
\draw  (v4) edge[->] (v5);
\draw  (v5) edge[->] (v7);
\draw  (v6) edge[double equal sign distance] (v7);
\node (v8) at (-5.5,2.5) {$\GCdim_{cat}(\mathcal C(M))$};
\draw  (v2) edge[->] (v8);
\draw  (v8) edge[->] (v6);
\node (v10) at (-1.5,0) {$\cuplength(M)$};
\node (v9) at (2,0) {$\gentime_{C_*(\Omega M)}(\mathcal C(M))$};
\draw  (v3) edge[->] (v9);
\draw  (v9) edge[->] node[fill=white] {\cref{cor:genTimeCocore}}  (v5);
\draw  (v10) edge[->] (v4);
\draw  (v10) edge[->] node[below] {\cref{prop:lower_bound_cuplength}} (v9);
\draw (v2) edge[bend right=20, <->, dashed] (v3);
\draw (v1) edge[double equal sign distance, dashed] (v10);
\end{tikzpicture}     \caption{Relating various forms of measures of complexity for categories and manifolds (over a field). Here, arrow $N_1 \to N_2$ indicates that $N_1 \leq N_2$. We conjecture that relations hold along the dashed edges.}
    \label{tab:comparingDimensions}
\end{table}
\Cref{tab:comparingDimensions} encapsulates some classical relations and the bounds given by \cref{mainthm:cotangent}. 
In addition to the invariants introduced in the previous section, this diagram includes the generation time of $\mathcal{C}(M)$ by $C_*(\Omega M)$ as a module over itself, and the triangulated good covering dimension which is defined in \cref{subsec:lsUpperBound} and is analogous to the triangulated LS category where instead the components of the diagram themselves are required to have zero Rouquier dimension. 
These comparisons to topology naturally lead one to expect an inequality involving $\LS(\mathcal{C})$ and $\Ddim(\mathcal{C})$ for any triangulated category $\mathcal{C}$. 
\begin{mainquestion}
    \label{question:ddimvsLS}
    Is there a relationship between $\LS(\mathcal{C})$ and $\Ddim(\mathcal{C})$ valid for any triangulated category $\mathcal{C}$?
\end{mainquestion}

In addition, we conjecture that \cref{prop:lower_bound_cuplength} can be improved to a lower bound for the Rouquier dimension.

\begin{mainconj}[\cref{conj:cuplength}]
    If $M$ is a smooth manifold and $\cuplength(M)$ is the cuplength of $H^*(M; k)$ for any field $k$, then
        $$ \cuplength(M) \leq \Rdim(\mathcal{C}(M)). $$
\end{mainconj}

Turning back to the more general case of a Liouville manifold $X$, it would interesting to have examples where the categorical dimensions are large.

\begin{mainquestion} \label{mainq:examples}
    Are there Liouville manifolds with $\Rdim(X) > \dim(X)/2$? Further, are there examples where $X$ is Weinstein and/or polarized?
\end{mainquestion}

Of course, \cref{mainq:examples} can also be posed for diagonal dimension and triangulated Lusternik-Schnirelmann category. We have focused on the Rouquier dimension and formulated the question in terms of $\dim(X)/2$ as it is particularly relevant in the context of homological mirror symmetry \cite{kontsevich1995homological} and Orlov's conjecture \cite{orlov2009remarks} that 
$$ \Rdim(D^b\Coh(Y)) = \dim(Y) $$
for any smooth quasi-projective scheme $Y$. The interplay between homological mirror symmetry and Orlov's conjecture as well as many known instances are detailed in \cite[Section 1.2]{bai2021rouquier}. In fact, \cite{bai2021rouquier} used \eqref{eq:bclowd} and known instances of homological mirror symmetry to establish new cases of Orlov's conjecture. In addition, the methods for proving the first part of \cref{mainthm:cotangent} have been translated to the mirror in the case $M$ is a torus to give a resolution of the diagonal for all toric varieties \cite{hanlon2023resolutions} which implies Orlov's conjecture for all toric varieties. One can also prove Orlov's conjecture directly using \cref{eq:bccotangent} and homological mirror symmetry for toric varieties as in \cite{favero2023rouquier}.

As we know of no compelling reason for there to be an obstruction to the existence of polarized Weinstein manifolds with highly singular mirror varieties, it seems reasonable to believe that \cref{mainq:examples} has a positive answer. On the other hand, \eqref{eq:bcgeneral}, \eqref{eq:bclowd}, and \cref{mainthm:embed} seem to suggest that the answer could be in the negative, which would be particularly interesting in the context of Orlov's conjecture.
\subsection{Outline}

The remainder of the paper is organized as follows. After giving brief motivation from manifold topology, \cref{sec:categorical} introduces the notions of categorical dimension that we consider along with some of their basic relations. In particular, the precise definitions of Rouquier dimension, diagonal dimension, and triangulated LS category appear in \cref{subsec:rdimbackground}, \cref{subsec:ddimbackground}, and \cref{subsec:triangLS}, respectively. \cref{sec:symplecticTools} introduces the necessary background on wrapped Fukaya categories and develop the symplectic tools that we use to bound the categorical invariants. In \cref{sec:applications}, we use this technology to obtain geometric bounds on categorical dimensions of wrapped Fukaya categories generally working first with $T^*M$ before generalizing to Weinstein manifolds in each case. We then use embeddings into cotangent bundles to bound the Rouquier dimension of wrapped Fukaya categories in \cref{sec:embedding}. Finally, \cref{subsec:conjectures} contains various conjectures and questions on extensions of the results obtained in the previous sections. 

\subsection{Acknowledgements}

The authors thank Mohammed Abouzaid, Denis Auroux, Laurent C\^{o}t\'{e}, Alexey Elagin, Yakov Eliashberg, David Favero, Sheel Ganatra, Jesse Huang, Wenyuan Li, John Pardon, Daniel Pomerleano, Nick Sheridan, and Abigail Ward for useful discussions.
We are also grateful to the anonymous referee for several suggestions and comments that improved the paper.
This is the ``A-side" of a project (whose ``B-side" is \cite{hanlon2023resolutions}) that began at the workshop ``Recent developments in Lagrangian Floer theory" at the Simons Center for Geometry and Physics, Stony Brook University. We thank the Simons Center and the workshop organizers for a stimulating scientific environment. In addition, part of this project was completed while AH and JH were participants in the SQuaRE program at the American Institute for Mathematics. We thank AIM for their support.

AH was supported by  NSF RTG grant DMS-1547145 and by the Simons Foundation (Grant Number
814268 via the Mathematical Sciences Research Institute, MSRI). JH was supported by EPSRC Grant EP/V049097/1.

\section{Categorical setup} \label{sec:categorical}
    In this section, we define the various categorical invariants that appear in the text. 
\subsection{Context: constructing a manifold ``efficiently"}
Before we discuss the measures of the complexity of a category, we discuss some ways to measure the complexity of a manifold that inspire how we interpret the various categorical dimensions.

We understand that the ``complexity'' of $M$ should measure how we can construct the manifold in an ``efficient'' manner.
Two natural ways to present a smooth manifold are via an open cover or via a handlebody decomposition.
\subsubsection{Open covers}
When covering a space, we have several different metrics of efficiency. First, we should assume that we cover the space with ``simple'' pieces. Secondly, we should try to cover our space with a minimal number of pieces or with the smallest amount of overlap. By varying these parameters, we get different measures of the complexity of a topological space.

A classical measure of the complexity of the space is the Lebesgue covering dimension, which is the smallest number $n$  so that every cover admits a refinement of ply $n+1$.
The ply of a cover is the smallest $k$ so that $\bigcap_{i\in I} U_i=\emptyset$ whenever $|I|>k$.
For a smooth manifold $M$, this is always the dimension of $M$. Observe that the Lebesgue covering dimension is not a homotopy invariant.
The Lebesgue covering dimension on metric spaces provides an upper bound for the good covering dimension, which is the minimal depth $k$ over all good covers of $M$. 
The existence of such a cover implies that the \v{C}ech cohomology groups of the constant sheaf are supported only in degrees $0$ to $k$. 

Another interesting measure of covering is the Lusternik-Schnirelmann category $\LS(M)$, which is the smallest number of open sets $U_i$ that cover $M$ such that the inclusions $\phi_i: U_i \hookrightarrow M $ are all null-homotopic. We note that the open set $U_i$ may be disconnected and count as a single open set. It is a theorem that $\LS(M)$  depends just on the homotopy type of $M$ \cite{james1978category}. It is known that the cuplength (with any coefficients) of $M$ is a lower bound for $\LS(M)$.

\subsubsection{Handlebody Decompositions} When presenting a space via a handlebody or a CW-decomposition, one can measure the complexity by trying to minimize the number of gluings or handle attachments required to build the space. 
For a manifold $M$, handle decompositions are described by Morse functions, where each handle corresponds to a critical point. Using Morse cohomology, we obtain the bound
\[\sum_i b_i(M) \leq \min_{f:M\to \RR \text{ Morse}}|\Crit(f)| \]
showing that the Betti numbers of $M$ provide a lower bound for the number of critical points of a Morse function. One could also ask for the minimal number of simultaneous handle attachments that are required to construct $M$, which corresponds to the number of critical values of $f$. We have bounds 
\[\cuplength(H^*(M)) \leq \min_{f:M\to \RR \text{ Morse }} |\critval(f)| \leq \dim(M)\]
where the lower bound follows from the fact that the cup product (with a non-minimal critical point) increases the critical value, and the upper bound follows from the existence of self-indexing Morse functions. 

In the world of CW complexes, the cone length of a space $M$ is defined to be the smallest value $k$ so that we have an iterated cofiber sequence
\[N_i\to M_i\to M_{i+1}\]
so that we start with $M_0=\{\bullet\}$ a point and end with $M_k=M$. The definition depends on which choices we allow for the $\{N_i\}$; Cornea \cite{cornea95lscategory} considers the case where you are allowed to attach along $i$-fold suspensions $Z_i=\Sigma^iU_i$ of arbitrary spaces $U_i$.
When $M$ is a manifold, it admits a CW presentation of length $\dim M$ showing that the cone-length of a manifold is at most $\dim (M)$. 
\subsubsection{Relations}
The relations between the bound on complexity that one obtains from covering versus handlebody decomposition are included in \cref{tab:comparingDimensions}. A more restricted version of the LS category (called the strong LS category) was shown to be equal to the cone length by \cite{cornea95lscategory}. It differs from the LS category by at most one.  The LS category is known to be a lower bound for the minimum number of critical points of a function with isolated critical points \cite{fox1941lusternik}. Note that the number of critical values of a Morsification of a degenerate critical point can be arbitrarily large \cite[Corollary 5.25]{pears1994degenerate}. This shows that locally one can find a large difference between the LS-category and the minimal number of critical values of a Morse function. However, we did not find a result in the literature producing this difference globally.
\begin{question}
    Does there exist a sequence of \emph{closed} manifolds $M_n$ such that the quantity 
    \[\frac{\min_{\substack{f:M_n\to \RR\\ \text{$f$ is Morse}}} |\critval(f)|}{\min_{\substack{g:M_n\to \RR\\ \text{$g$ has isolated critical points}}} |\critval(g)|} \]
    is unbounded?
\end{question}
Additionally, we could not find a bound (or an example suggesting the nonexistence of a bound) between the LS category and the minimum number of critical values of a function with isolated critical values. A detailed discussion is given in \cite[Section 5]{pears1994degenerate}.
Returning to cuplength, we note that the cuplength of $H^*(M)$ is a lower bound for the LS category of $M$ \cite[Theorem 28.1]{fox1941lusternik}. \subsection{Rouquier Dimension} \label{subsec:rdimbackground}
We now turn to the definitions and basic properties of the categorical invariants that we study. We work with $A_\infty$ categories, as opposed to simply triangulated categories. Several of the invariants we consider require an $A_\infty$ (or dg) enhancement, and $A_\infty$ categories arise naturally in our applications. For the remainder of \cref{sec:categorical}, we will denote by $\mathcal{C}$ an $A_\infty$ category linear over a commutative ring $R$ and set  $\Perf \mathcal{C}$ to be the idempotent-completed pre-triangulated closure of $\mathcal{C}$. Similarly, we let $\Perf R$ be the dg category of perfect complexes of $R$-modules. 

We begin by defining the Rouquier dimension after \cite{rouquier2008dimensions} and give some mild generalizations of the definition.
Given two full subcategories $ \mathcal{G}_1,  \mathcal{G}_2$ of $\Perf \mathcal C$, we define $\mathcal{G}_1 * \mathcal{G}_2$ to be the full subcategory on objects $G$ that sit in a distinguished triangle of the form $G_1 \to G \to G_2 \to G_1[1]$ with $G_i \in \mathcal{G}_i$. For any full subcategory $\mathcal{G}$ of $\Perf \mathcal{C}$, we define $\langle \mathcal{G} \rangle$ to be the smallest full subcategory of $\Perf \mathcal C$ that contains $\mathcal G$ and is closed under quasi-isomorphisms, direct summands, direct sums, and shifts.

We will need an alternate version of this construction which is not as sensitive to the choice of ground ring $R$. 
For any $P \in \Perf R$ and $A \in \Perf \mathcal C$, we can form the tensor product $P \tensor A \in \Perf \mathcal C$. 
We define $\langle \mathcal{G}\rangle^R$ to be the smallest full subcategory of $\Perf \mathcal C$ that contains $\mathcal G$ and is closed under direct summands, sums, isomorphisms, shifts, and tensor products with $\Perf R$. 
When the ground ring $R$ is a field, observe that $\langle \mathcal G \rangle = \langle \mathcal G\rangle^R$ as tensoring with a finite dimensional vector space can be represented by a direct sum. 

Now, given any full subcategory $\mathcal G$ of $\Perf \mathcal C$, we construct a sequence of full subcategories 
$$\langle \mathcal G \rangle_0, \langle \mathcal G \rangle_1, \langle \mathcal G \rangle_2, \ldots $$ 
inductively as follows. 
We define
$\langle \mathcal \mathcal{\mathcal{G}} \rangle_{0}=0$ 
and set 
\[\langle \mathcal{G} \rangle_k = \langle \langle \mathcal{G} \rangle_{k-1} * \langle \mathcal{G} \rangle \rangle.\]
If $\mathcal{G}$ and $\mathcal{E}$ are full subcategories of $\mathcal{C}$ and there is a $k$ such that $\mathcal{E} \subset \langle \mathcal G \rangle_k$, we define 
$$\gentime_{\mathcal{G}} (\mathcal{E}) = \min \left\{ k \st \mathcal{E} \subset \langle \mathcal G \rangle_{k+1} \right\} $$
and call this the \emph{generation time} of $\mathcal{E}$ by $\mathcal{G}$.
We say that $\mathcal{G}$ \emph{generates}\footnote{In some literature, particularly on $A_\infty$ categories in symplectic geometry, what we call generation is called \emph{split-generation} to demarcate that we have allowed splitting off of direct summands. Further, what we call resolving an object is sometimes referred to as generating the object.} 
$\mathcal{C}$ if
$$ \Perf \mathcal{C} = \bigcup_{k \geq 0} \langle \mathcal{G} \rangle_k. $$
We can apply the same construction to any object $G$ of $\mathcal{C}$ by conflating it with the full subcategory consisting of that object. The same construction can be used to define $R$-generation time, where $\langle \mathcal{G} \rangle_k^R = \langle \langle \mathcal{G} \rangle_{k-1}^R * \langle \mathcal{G} \rangle^R\rangle^R$ and 
$$\gentime_{\mathcal{G}}^R (\mathcal{E}) = \min \left\{ k \st \mathcal{E} \subset \langle \mathcal G \rangle_{k+1} ^R\right\}. $$ 

Finally, we can define the Rouquier dimension.

\begin{definition}[Rouquier dimension and variations]
    The \emph{Rouquier dimension} of $\mathcal{C}$, $\Rdim(\mathcal{C})$, is the smallest non-negative integer $k$ such that there is an object $G$ of $\Perf \mathcal{C}$ satisfying $\gentime_G(\mathcal C) = k$.
    If no such $k$ exists, we set $\Rdim(\mathcal{C}) = \infty$. 
    
    For an $A_\infty$ functor $j: \mathcal B\to \mathcal C$, we write $\Rdim(j)$ for the smallest non-negative integer $k$ so that there exists an object $G$ of $\Perf \mathcal C$ satisfying $\gentime_G(\Im(j))=k$ where $\Im(j)$ is the essential image (which may itself not be triangulated).

    If $\mathcal C$ is linear over $R$, define the $R$-Rouquier dimension to be the smallest $k$ such that there is an object $G$ of $\Perf \mathcal C$ such that $\gentime_G^R(\mathcal C)=k$; we similarly define $\Rdim^R(j)$ for an $A_\infty$ functor $j: \mathcal B\to \mathcal C$.
\end{definition}
\begin{example} \label{ex:RdimZ}
    If $R$ is a field, the generation time and Rouquier dimension agree with $R$-generation time and $R$-Rouquier dimension. However, for general rings this may not be the case. For example, when $R=\ZZ$, we have that $\Rdim(\Perf \ZZ)=1$ by the classification of finitely generated abelian groups. By definition, $\Rdim^R(\Perf \ZZ)=0$ as the smallest category containing $\ZZ$ and closed under tensor products by objects of $\Perf \ZZ$ is $\Perf \ZZ$ itself.
\end{example}

\begin{rem} 
    The Rouquier dimension can be defined for triangulated categories without any extra structure as in \cite{rouquier2008dimensions}, and $\Rdim(\mathcal C)$ is an invariant of the triangulated category $H^0 (\Perf C)$. However, $\Rdim^R(\mathcal{C})$ relies on the $A_\infty$ structure to define the tensor product and we do not know if there is a triangulated analog.
\end{rem}

\begin{question} \label{q:rdimvsrdimr}
    What is the relation between $\Rdim(\mathcal C)$ and $\Rdim^R(\mathcal C)$?
\end{question}
Because $\langle G\rangle^R\subset \langle G\rangle_{\gentime_R(\Perf R)}$, we can obtain the na\"ive bound
\begin{equation}
    \Rdim^R(\mathcal C) \leq \Rdim(\mathcal C)\leq (\gentime_R(\Perf R)+1)\cdot \Rdim^R(\mathcal C).
    \label{eq:badRdimRbound}
\end{equation}
However, evidence (\cref{prop:weakestDdimbound}) suggest that under some appropriate conditions we might expect $\Rdim(\mathcal C) \leq \Rdim^R(\mathcal C)+\gentime_R(\Perf R)$. 

We use the symbol $[G_1\xrightarrow{f_1}G_0]$ to denote the mapping cone of a morphism $f_1 \colon G_1\to G_0$ in $\Perf \mathcal C$. Furthermore, the notation $[G_{k}\xrightarrow{f_k} G_{k}\to \cdots \xrightarrow{f_1} G_0]$ represents an iterated mapping cone in $\Perf \mathcal C$, where each $f_i$ corresponds to a morphism from $G_{i}$ to $[G_{i-1}\to\cdots \to G_0]$. We define the \emph{resolution time} of $E\in \mathcal C$ by $\mathcal G$ to be the smallest $k$ such that $E\cong[G_{k}\to \cdots \to G_0]$ with $G_i$ belonging to $\mathcal G$, and write $\Res_{\mathcal G}(E)=k$ given the existence of such a $k$. Observe that for any generating subcategory $\mathcal G$, we have
\[ \gentime_{\mathcal G}(E)\leq \Res_{\mathcal G}(E)\]
for all $E \in \mathcal C$.
It is in general harder to resolve an object than it is to generate it as resolving does not allow splitting off of direct summands. 

We now give some preparatory lemmas. These are well-known when working with $\Rdim$ (see, for instance, \cite[Lemmas 2.4 and 2.5]{ballard2012hochschild}), but we present an argument here to make it clear they apply to $\Rdim^R$ as well.

\begin{lemma}
    Let $j:\mathcal C\to \mathcal B$ be an $A_\infty$ functor between $R$-linear $A_\infty$ categories, and let $\mathcal C_1, \mathcal C_2$ be two full subcategories of $\Perf \mathcal C$ so that $\langle \mathcal C_i\rangle=\mathcal C_i$.
    Then 
    \[\langle j(\langle \mathcal C_1 * \mathcal C_2\rangle )\rangle= \langle j(\mathcal C_1)*j(\mathcal C_2)\rangle.\] Similarly, 
    \[\langle j(\langle \mathcal C_1 * \mathcal C_2\rangle^R )\rangle^R= \langle j(\mathcal C_1)*j(\mathcal C_2)\rangle^R\]
\end{lemma}
\begin{proof}
    We prove the first equality as the second follows from the same argument. Let $B\in \langle j(\langle \mathcal C_1 * \mathcal C_2\rangle )\rangle$. As $\langle \mathcal C_1 * \mathcal C_2\rangle$ is closed under direct sums and shifts, $j(\langle \mathcal C_1 * \mathcal C_2\rangle)$ is closed under direct sums and shifts. Therefore, there exists $A\in\langle \mathcal C_1* \mathcal C_2\rangle$  so that $B$ is a direct summand of $j(A)$. Since the $\mathcal C_i$ are closed under shifts and summands so is $\mathcal C_1 * \mathcal C_2$, so $A$ is a direct summand of some $[ C_1\to C_2 ]$. Since $j$ is a triangulated functor, $j(A)$ is a direct summand of $[j(C_1)\to j(C_2)]$. In particular, $B$ is an object of $\langle j(\mathcal C_1)*j(\mathcal C_2)\rangle$. 

    The reverse direction is analogous.
    Let $B\in \langle j(\mathcal C_1)*j(\mathcal C_2)\rangle$; then $B$ is a summand of $\cone(j(C_1)\to j(C_2))$, and therefore a summand of $j(\cone(C_1\to C_2))$.
\end{proof}
\begin{corollary}
    Let $\pi: \mathcal C\to \mathcal B$ be an essentially surjective $A_\infty$ functor. Then $\Rdim(\mathcal B)\leq \Rdim(\mathcal C)$, and $\Rdim^R(\mathcal B) \leq \Rdim^R(\mathcal C)$
\end{corollary}

 \subsection{Diagonal Dimension} \label{subsec:ddimbackground}

In some cases, one can understand Rouquier dimension by resolving or generating the diagonal bimodule, and the generation time of the diagonal itself provides another interesting invariant. Let $\Delta$ be the diagonal $\mathcal{C}^{op} \otimes \mathcal{C}$-bimodule. We set $\mathcal{P}$ to be the subcategory of perfect  $\mathcal{C}^{op} \otimes \mathcal{C}$-bimodules, that is, $\mathcal{P}$ is the idempotent closure of the subcategory of representable bimodules. We will often abuse notation and write $L \otimes K$ for the perfect bimodule represented by $L, K \in \mathcal{C}$.

\begin{definition} \label{def:diagonaldim}
    The diagonal dimension $\Ddim(\mathcal C)$ is the smallest non-negative integer $k$ such that there is a $G \in \mathcal{P}$ with $\gentime_G \Delta = k$.     
    If no such $k$ exists, we set $\Ddim(\mathcal C) = \infty$. 
\end{definition}

When $\Ddim(\mathcal{C}) < \infty$, $\mathcal{C}$ is called homologically smooth \cite{kontsevich2008notes}.
\cref{def:diagonaldim} appears in \cite[Definition 2.15]{ballard2012hochschild} in the context of derived categories of algebraic varieties and in \cite[Definition 4.1]{elagin2021three} in the context of dg algebras. 
\cite[Lemma 2.16]{ballard2012hochschild} shows that the diagonal dimension provides an upper bound for the Rouquier dimension of the derived category of coherent sheaves of a proper algebraic variety over a field. 
An analogous statement holds for general $A_\infty$ categories, which we will now discuss. Along the way, we examine how conditions on the category, such as properness, the base ring, and generation strength (resolving vs. generation), affect the relationship between these invariants.

\begin{table}[]
    \centering
    \begin{tabular}{c|cc|cc|}
                                                         & \multicolumn{2}{c|}{$\mathcal C$ Proper}                                            & \multicolumn{2}{c|}{$\mathcal C$ Non-Proper}                                                           \\
                                                         & \multicolumn{1}{c|}{$R$ field}                                           & $R$ ring & \multicolumn{1}{c|}{$R$ field}                     & $R$ ring                                          \\ \hline
    $\Res_{\mathcal G\otimes \mathcal G} \Delta=k$     & \multicolumn{1}{c|}{$\Res_{\mathcal G}\mathcal C=k$}                     & $\Res_{\mathcal G}\mathcal C\leq k+\Res_R(\Perf R)$       & \multicolumn{1}{c|}{\multirow{2}{*}{$\gentime_{\mathcal G} \mathcal C\leq k$} }&    \multirow{2}{*}{$\gentime_{\mathcal G}\mathcal C\leq k+\gentime_R(\Perf R)$ }                                               \\ \cline{1-3}
    $\gentime_{\mathcal G\otimes \mathcal G} \Delta= k$ & \multicolumn{1}{c|}{$\gentime_{\mathcal  G}\mathcal C=k$}   	    & $\gentime_{\mathcal  G}\mathcal C\leq k+\gentime_R(\Perf R)$       &  \multicolumn{1}{c|}{}                       & \\ \hline
    \end{tabular}
    \caption{Bounding generation of the category from the generation of the diagonal.}
    \end{table}
The strongest generation result is where we assume the category is proper and we work over a field.
\begin{prop}
\label{prop:diagRestoRes}
If $R$ is a field, $\mathcal C$ is proper, and $\Delta$ is resolved in time $m$ by a finite full subcategory of product objects $\mathcal H\tensor \mathcal G \subset \mathcal{P}$, then every object of $\mathcal C$ is resolved by $\mathcal G$ in time $m$.
\end{prop}
\begin{proof}[Sketch of proof] Suppose that $\{L_i\}$ and $\{K_i\}$ are the finite sets of objects of $\mathcal G$ and $\mathcal H$, respectively. Consider an arbitrary object 
$T$. Then $T \cong \Delta\circ T$, which is a twisted complex whose terms are sums of $H^\bullet(\hom(K_i, T)) \otimes L_i$. Since $\mathcal C$ is proper and defined over a field, each $H^\bullet\hom(K_i, T)$ is a finite-dimensional graded vector space. Thus, $H^\bullet(\hom(K_i, T)) \otimes L_i$ is isomorphic to a direct sum of (shifted) copies of $L_i$. As a result, $T$ is resolved in time $m$ by the collection $\{L_i\}_{i=1}^k$.
\end{proof}

Without properness, a resolution of the diagonal does not necessarily lead to a resolution of every object. However, generating the diagonal in time $k$ by product objects implies time $k$ generation of every object over a field even without assuming properness.

\begin{prop}
Suppose that $R$ is a field. Let $\mathcal G$ and $\mathcal H$ be finite full subcategories of $\mathcal C$. If $\Delta \in \left \langle \mathcal G\tensor \mathcal H\right \rangle_k$, then $\left \langle \mathcal G\right \rangle_k = \mathcal{C}$.  In particular, $\Rdim(\mathcal C) \leq \Ddim(\mathcal C)$. 
\label{prop:diagGenerationToGeneration} 
\end{prop}
\begin{proof} This follows from \cite[Lemma 2.13]{bai2021rouquier} noting that \cite[Lemma 2.2(2)]{bai2021rouquier} shows that $\Delta \in \left \langle \oplus L_i \otimes K_i\right \rangle_k$ is sufficient to produce a resolution of a bimodule of which $\Delta$ is a summand. To show that only finite direct sums of the generators are possible, one applies \cite[Lemma 3.14]{rouquier2008dimensions}.
\end{proof}

\begin{rem}
    The diagonal dimension can be strictly larger than the Rouquier dimension even in geometric contexts.    
    For example, all smooth projective curves with genus at least 1 have a derived category with diagonal dimension $2$ \cite{olander2021diagonal} while the Rouquier dimension is $1$ \cite[Theorem 6]{orlov2009remarks}.
\end{rem}

When working over an arbitrary commutative ring, the diagonal dimension does not bound the Rouquier dimension.  For example, $\Rdim(\Perf \ZZ) = 1$ as noted in \cref{ex:RdimZ}. On the other hand, the diagonal dimension is $0$ since the diagonal bimodule is the perfect bimodule $\mathbb{Z} \otimes \mathbb{Z}$. This can be corrected with an error term arising from the generation time of $\Perf(R)$ by $R$.

\begin{prop} \label{prop:weakestDdimbound}
For any commutative ring $R$ and $A_\infty$ category $\mathcal{C}$,  we have 
$$
\Rdim(\mathcal C) \le \Ddim(\mathcal{C})  + \gentime_R(\Perf R).
$$
\end{prop}
\begin{proof}
We first assume that $\Delta_\mathcal{C}$ is resolved in time $n = \Ddim(\mathcal{C})$, i.e., there is a twisted complex with terms $K_i \otimes L_i$ for $0 \le i \le n$ quasi-isomorphic to $\Delta_\mathcal{C}$. Let $\mathcal{C}-\mod$ be the category of all $A_\infty$-modules over $\mathcal{C}$ and let $R-\mod$ be the dg category of all complexes of $R$ modules.
Then for any object $T$ of $\mathcal{C}$, $T$ is equivalent to a twisted complex $H^\bullet \hom(K_i, T) \otimes L_i$ in 
$\mathcal{C}-\mod$. Here, $H^\bullet \hom(K_i, T)$ may be a general object of $R-\mod$ and not in $\Perf R$.
However, we observe that $R-\mod$ is the Ind-completion of $\Perf R$, and the subcategory of compact objects of $R-\mod$ is $\Perf R$. Therefore, the twisted complex for $T$ with terms $H^\bullet \hom(K_i, T) \otimes L_i$ is a filtered colimit of complexes $T_j$ with terms $P_{i,j} \otimes L_i$ where $P_{i,j} \in\Perf(R)$. 
Further, $T$ is a compact object of $\mathcal{C}-\mod$ since it lies in $\mathcal{C}$. 
Therefore, the identity morphism $\id_T$ of $T$ factors through one of the complexes $T_j$  with terms $P_{i,j} \otimes L_i$ where $P_{i,j} \in \Perf R$. Therefore, $T$ is a summand of $T_j$. Observe that the terms of $P_{i, j}\otimes L_i$ lie in $\langle \mathcal G\rangle_{\gentime_R(\Perf R)}$, from which we obtain the bound 
\[\gentime_{L_i}(T) \le \Ddim(\mathcal{C}) +\gentime_R(\Perf(R)).\]
If $\Delta_{\mathcal C}$ is generated in time $k$, instead of resolved, the argument from \cref{prop:diagGenerationToGeneration} applies yielding the same bound.
\end{proof}

\subsection{Triangulated LS category of a category} \label{subsec:triangLS}
We now define notions of categorical complexity in terms of presenting an $A_\infty$ category as a homotopy colimit. 
We first recall some constructions from \cite[Appendix A.4]{ganatra2018sectorial}.
Let $\Sigma$ be a diagram.
The height $h: \Sigma\to \NN$ associates to each $\sigma$ the length of the longest path 
\[ \sigma_0\leftarrow \cdots \leftarrow \sigma_{h}=\sigma \]
ending at $\sigma$. The \emph{depth} of the diagram, which we denote by $D(\Sigma)$, is the maximum value of $h$ (and is $\infty$ if no maximum is achieved).

Let $\mathcal F:\Sigma\to \Cat$ be a diagram of $A_\infty$ categories where $\Sigma$ is a posetal category. The \emph{Grothendieck construction} associates to $\mathcal{F}$ the $A_\infty$ category $\Grot(\mathcal F: \Sigma\to \Cat)$ whose objects are pairs $(\sigma, C)$, with $\sigma\in \Sigma$ and $C\in \Ob(\mathcal F(\sigma))$. Morphisms in the Grothendieck construction are pairs $(F, f: F(C_\sigma)\to C_\tau)$. Given a functor $j: \Grot(\mathcal F:\Sigma\to \Cat)\to \mathcal C$, we obtain functors $j_\sigma: \mathcal F(\sigma)\to \mathcal C$ for all $\sigma\in \Sigma$. Let $A$ denote the set of morphisms of the form $(F, \id: F(C_\sigma)\to F(C_\sigma))$ in $\Grot(\mathcal F:\Sigma\to \Cat)$.  In \cite{ganatra2018sectorial}, the homotopy colimit of the diagram is defined to be $\hocolim(\mathcal F:\Sigma\to \Cat):= \Grot(\mathcal F: \Sigma\to \Cat)[A^{-1}]$. We say that a diagram $\mathcal F:\Sigma\to \Cat$ presents $\mathcal C$ as a homotopy colimit if we have a quasi-equivalence 
\[ j: \Perf \hocolim(\mathcal F:\Sigma\to \Cat) \xrightarrow{\sim} \Perf \mathcal C .\]

Define the \emph{triangulated good covering dimension} $\GCdim_{cat}(\mathcal C)$ to be the minimum over all presentations of $\mathcal C$ as a homotopy colimit where $\Rdim(\mathcal F(\sigma))=0$ for all $\sigma \in \Sigma$. A specialization of \cite[Lemma 2.15]{bai2021rouquier}, shows that $\Rdim(\mathcal C)\leq  \GCdim_{cat}(\mathcal C)$ over a field. 

\begin{definition} \label{def:triangulatedLS}
    Let $R$ be a commutative ring and let $\mathcal F: \Sigma\to \Cat$ be a diagram of $A_\infty$ categories. A presentation of $\mathcal C$ as a homotopy colimit given by $j: \Grot(\mathcal F:\Sigma\to \Cat)\to \mathcal C$ is an $R$-\emph{Lusternik-Schnirelmann covering} of $\mathcal C$ if $\Rdim^{R}(j_\sigma)=0$ for all $\sigma$. We say that it is a \emph{Lusternik-Schnirelmann covering} if $\Rdim(j_\sigma)=0$ for all $\sigma$.
    We define the \emph{$R$-triangulated LS category}  to be 
    \[\LS_{cat}^R(\mathcal C):=\min_{\text{$R$-LS covers of $\mathcal C$}} D(\Sigma).\]
    and the \emph{triangulated LS category} to be\[\LS_{cat}(\mathcal C):=\min_{\text{LS covers of $\mathcal C$}} D(\Sigma).\]
\end{definition}

\begin{prop}
    If $\mathcal C$ is an $A_\infty$ category linear over $R$, then
    \begin{align*}\Rdim^R(\mathcal C) \leq \LS_{cat}^R( \mathcal C) && \text{ and } &&  \Rdim(\mathcal C)\leq \LS_{cat}(\mathcal C).\end{align*}
    \label{prop:TLSCatToRdim}
\end{prop}

\begin{proof}
    We will prove the statement for $R$-triangulated LS category, as the proof is analogous for the triangulated LS category. 
    We prove a more general relation, which is that given any functor $j: \Grot(\mathcal F:\Sigma\to \Cat) \to \mathcal{C}$, we have 
    \[\Rdim^R(j)\leq D(\Sigma)+\sum_{i=0}^{D(\Sigma)} \max_{\sigma \st h(\sigma)=i} \Rdim^R(j_\sigma: \mathcal F(\sigma)\to \mathcal C).\]
    where $h(\sigma)$ is the length of the chain in $\Sigma$ starting at $\sigma$. The formula is similar to \cite[Lemma 2.10]{bai2021rouquier} except that we are working with essential images.
    
    We induct on the depth of $\Sigma$.
    Suppose that $D(\Sigma) = k$ and the statement holds for all $\Sigma'$ with $D(\Sigma')=k-1$. Let $\mathcal F: \Sigma\to \Cat$ 
    be a diagram and $j: \Grot(\mathcal F:\Sigma\to \Cat)\to \mathcal C$ be a functor.
    Write $\Sigma= \langle \Sigma^k, \Sigma^{<k}\rangle$, where $\Sigma^{<k}$ is the full subcategory of objects whose depth is less than $k$, and let $\mathcal F^{<k}: \Sigma^{<k}\to \Cat$ and $j^{<k}:\mathcal F(\sigma)\to \mathcal C$ to be the restriction of $\mathcal F:\Sigma\to \Cat$, $j$ to this subcategory.  
    Then consider the diagram $\Sigma'$, whose
    \begin{itemize}
        \item objects of $\Sigma'$ are  $\Ob(\Sigma^k)\cup\{\bullet\}$
        \item arrows of $\Sigma'$ are given by $\hom(\sigma,\bullet)=\bigcup_{\tau\in \Sigma^{<k}}\hom(\sigma, \tau)$.  
    \end{itemize}
    and the functor $\mathcal F': \Sigma'\to \Cat$ given by 
    \begin{itemize}
        \item $\mathcal F'(\sigma) = \mathcal F(\sigma)$ when $\sigma\in \Sigma^k$, and $\mathcal F'(\bullet) = \Grot(\mathcal F^{<k}:\Sigma^{<k}\to \Cat)$.
        \item When $f: \sigma\to \bullet$ corresponds to a morphism $\tilde f: \sigma\to \tau$ in $\Sigma$, the functor $\mathcal F'(f): \mathcal F'(\sigma)\to \mathcal F'(\bullet)$ is defined by the composition 
        \[\mathcal F'(\sigma)=\mathcal F(\sigma)\to \mathcal F(\tau)=\mathcal F'(\tau)\to \Grot(\mathcal F^{< k})=\mathcal F'(\bullet),\]
    \end{itemize}
    and consider the functor $j': \Grot(\mathcal F':\Sigma'\to \Cat)\to \mathcal C$  making the following diagram commute:
    \[
        \begin{tikzcd}
            \Grot(\mathcal F':\Sigma'\to \Cat) \arrow[equals]{r} \arrow{dr}{j'} & \Grot(\mathcal F:\Sigma\to \Cat)\arrow{d}{j} \\
            & \mathcal C
        \end{tikzcd}.
    \]
    Observe that $D(\Sigma')=1$ and $D(\Sigma^{<k})\leq k-1$. By induction, we have 
    \begin{align*}
        \Rdim^R(j)\leq&D(\Sigma')+\left(\max_{\sigma\in \Sigma^k}\Rdim^R(j_\sigma:\mathcal F(\sigma)\to \mathcal C) + \right)+\Rdim^R(j')\\ 
        \leq&  1+\left(\max_{\sigma\in \Sigma^k}\Rdim^R(j_\sigma:\mathcal F(\sigma)\to \mathcal C)\right)\\
        &+D(\Sigma^{<k})+\sum_{k=0}^{D(\Sigma^{<k})}\left(\max_{\sigma \st h(\sigma)=k} \Rdim^R(j_\sigma: \mathcal F(\sigma)\to \mathcal C) \right)\\
    \leq& D(\Sigma)+ \sum_{i=1}^{D(\Sigma)}\left(\max_{\sigma \st h(\sigma)=i} \Rdim^R(j_\sigma: \mathcal F(\sigma)\to \Cat)\right).
    \end{align*}
    
    It suffices therefore to prove the base case where the depth is 1. For each $j_\sigma: \mathcal F(\sigma)\to C$, let $G_\sigma\subset C$ be the element realizing the Rouquier dimension, i.e. $\langle G_\sigma\rangle_{\Rdim^R(j_\sigma)+1}\supset \Im(j_\sigma)$, and let $i_\sigma: \mathcal F(\sigma)\to \Grot(\mathcal F:\Sigma\to \Cat)$ be the inclusion. When $\Sigma=\sigma\to \tau$, we can generalize \cite[Lemma 2.3]{bai2021rouquier} to obtain
\begin{align*}
     \langle \Im(j:\Grot(\mathcal F(\sigma)\to \mathcal F(\tau)) )\rangle^R =& 
    \left\langle  j \left(\langle i( \mathcal F(\sigma) ) \rangle^R \ast \langle i(\mathcal F(\tau) )\rangle^R\right)\right\rangle^R\\
    =& \langle j(\langle  i (\mathcal F(\sigma)) \rangle^R) * j( \langle i(\mathcal F(\tau) )\rangle^R)\rangle\\    =& \langle \langle  j_\sigma (\mathcal F(\sigma)) \rangle^R \ast  \langle j_\tau(\mathcal F(\tau) )\rangle^R\rangle\\
    \subset& \langle G_\sigma\rangle^R_{k_1+1} \ast \langle G_\tau\rangle_{k_2+1}^R
    \subset  \langle G_\sigma\oplus G_\tau\rangle^R_{k_1+k_2+2}
\end{align*}
So that 
\begin{align*}
    \Rdim^R(\Grot(\mathcal{F}(\sigma)\to \mathcal{F}(\tau)))= &1+\Rdim^R(\mathcal F(\sigma))+\Rdim^R(\mathcal F(\tau))\\
    = &D(\Sigma)+\Rdim^R(\mathcal F(\sigma))+\Rdim^R(\mathcal F(\tau))
\end{align*}
as needed.
\end{proof}
 
\section{Symplectic tools: resolutions and LS-coverings in Liouville domains}
\label{sec:symplecticTools}
In this section, we will develop tools to interpret the previous categorical constructions from the viewpoint of symplectic geometry. 
First, we will establish notation and recall certain results about wrapped Fukaya categories of Liouville domains that we will need. In \cref{subsec:cobordism}, we review a geometric interpretation of resolving Lagrangian submanifolds in the wrapped Fukaya category using Lagrangian cobordisms. In particular, \cref{cor:Weinsteinresolve} provides an upper bound for the resolution time of a Lagrangian submanifold in a Weinstein domain by cocores of the skeleton. Finally, in \cref{subsec:sectorialLS}, we introduce a geometric analog of \cref{def:triangulatedLS}. 

\subsection{Background: Liouville domains and wrapped Fukaya categories}
In this section, we provide some background on Liouville domains and their wrapped Fukaya categories. The wrapped Fukaya category was first defined in \cite{abouzaid2010open} and has been extensively studied. We will use the version of the wrapped Fukaya category as defined in \cite{ganatra2017covariantly,ganatra2018microlocal}.

\begin{definition}[Liouville Domain and Manifold]
    A \emph{Liouville domain} is a pair $(X_0, \lambda)$ where
    \begin{itemize}
        \item $X_0$ is a compact manifold with boundary,
        \item $\lambda$ is a 1-form so that $\omega:=d\lambda$ is symplectic and,
        \item The Liouville vector field $Z$ satisfying $\iota_Z\omega = \lambda$ points outwards along the boundary of $X_0$.
    \end{itemize}
    A \emph{Liouville manifold} is an exact symplectic manifold $(X, \lambda)$ containing a Liouville domain $X_0$ as an exact symplectic submanifold such that the positive Liouville flow out from $\partial X_0$ is defined for all time and induces a diffeomorphism $X_0 \cup \left(\partial X_0 \times \RR_{\geq 0}\right) \to X$. 
\end{definition}

If $(X, \lambda)$ is a Liouville manifold, a Lagrangian submanifold $L\subset X$ is \emph{conical near the boundary} if it is parallel to $Z$ in a neighborhood of the boundary of $X$. The \emph{wrapped Fukaya category} $\mathcal W(X, \lambda)$, as described in \cite[Section 3]{ganatra2017covariantly}, is an $A_\infty$ category whose objects are supported on exact and conical Lagrangian submanifolds of $X$, morphisms are given by wrapped Floer cohomology groups with coefficients in a commutative ring $R$, and compositions are defined by counts of $J$-holomorphic polygons. For any compactly supported function $f$, we note that $\mathcal W(X, \lambda)$ and $\mathcal W(X, \lambda+df)$ are quasi-equivalent \cite[Lemma 3.41]{ganatra2017covariantly}. When the form $\lambda$ is understood, we will suppress it from our notation. In order to define $\mathcal{W}(X)$ as a $\ZZ$-graded category linear over an arbitrary ring $R$, one must choose extra grading/orientation data as in \cite[Section 5.3]{ganatra2018microlocal}. One way to specify this data is through a (stable) Lagrangian polarization on $X$. The cotangent bundle $T^*M$ of a smooth manifold has a canonical polarization given by the cotangent fibers.

The \emph{skeleton} (or core) $\LL_X$ of a Liouville manifold $X$, is the set of points of $X$ which do not escape under the Liouville flow. A \emph{Weinstein manifold} is a Liouville manifold whose Liouville vector field is gradient-like with respect to a proper Morse function. The Morse function allows a Weinstein manifold to be built from handles that all have index $\leq \dim(X)/2$, and the skeleton is the collection of attaching cores. See \cite{cieliebak2012stein}. In particular, $\LL_X$ is \emph{mostly Lagrangian} in the sense of \cite[Definition 1.7]{ganatra2018sectorial}.

\subsubsection{Homotopy colimit formula}

In order to glue wrapped Fukaya categories, \cite[Definition 2.4]{ganatra2017covariantly} introduces the notion of a \emph{Liouville sector}. A Liouville sector is a Liouville manifold $X$ with boundary $\partial X$; the boundary is equipped with the additional data of a function $I \colon \partial X \to \RR$ satisfying certain properties allowing for holomorphic curves to be constrained to lie away from the boundary of $X$. 
An \emph{inclusion} of a Liouville sector is a proper map $i: X \to X'$ such that $i^*\lambda' = \lambda + df$ for some compactly supported functions $f$. 
A \emph{Weinstein sector} is a Liouville sector whose corresponding sutured Liouville manifold $(\overline X, F_0)$ from \cite[Lemma 2.32]{ganatra2017covariantly} consists of a Weinstein manifold $\overline X$ with Weinstein hypersurface $F_0$ (up to deformation). 

\begin{example}
    Let $U\subset M$ be any codimension 0 submanifold with boundary. Then $T^*U$ is an example of a Weinstein sector, and $T^*U\subset T^*M$ is a proper inclusion. 
\end{example}

Liouville sectors also have well-defined wrapped Fukaya categories, which we will also denote by $\mathcal W(X)$. The inclusion $X'\into X$ of a Liouville sector induces a functor $\mathcal W(X')\to \mathcal W(X)$ as defined in \cite[Section 3.6]{ganatra2017covariantly}.
In \cite[Theorem 1.35]{ganatra2018sectorial}, it is shown that Weinstein sectors satisfy a descent property. More precisely, if $X=\bigcup_i X_i$ is a Weinstein sectorial covering \cite[Definition 1.32]{ganatra2018sectorial} of a Liouville sector $X$ by subsectors, the wrapped Fukaya category can be computed as the homotopy colimit of the induced diagram of wrapped Fukaya categories. That is, if we set $X_I=\cap_{i\in I} X_i$, there is a fully faithful functor
\begin{equation} \label{eq:hocolim}
    \hocolim(\mathcal W(X_I)) \to \mathcal W(X)
\end{equation}
whose image resolves every object.

\subsubsection{K\"unneth formula and Diagonal bimodule}
\label{eq:cotangentCorrespondences}
Additionally, partially wrapped Fukaya categories satisfy a K\"unneth formula, in that there is a fully faithful functor 
\begin{equation} \label{eq:kunneth}
    \mathcal W(X_1)\tensor \mathcal W(X_2)\to \mathcal W(X_1\times X_2)
\end{equation}
for any Liouville sectors $X_1, X_2$ by \cite[Theorem 1.5]{ganatra2018sectorial}.
When $X_1$ and $X_2$ are Weinstein, every object of $\mathcal{W}(X_1 \times X_2)$ is resolved by objects in the image of \eqref{eq:kunneth} by \cite[Corollary 1.18]{ganatra2018sectorial}.
Under this identification, the diagonal Lagrangian $\Delta \subset \mathcal W( X^- \times X)$ is identified with the diagonal bimodule as shown independently in \cite[Theorem 7.2]{bai2021rouquier} and \cite[Proposition 1.6]{ganatra2018sectorial} where $X^-$ is the Liouville manifold with the same underlying manifold as $X$ but Liouville form $-\lambda$.

In the context of cotangent bundles, there is a symplectic involution ${}^-(-): (T^*M)^-\times T^*M\to T^*M\times T^*M=T^*(M\times M)$, which is given by multiplication by $-1$ on the fiber coordinate of the first factor 
\begin{align*}
    {}^-(-):T^*M_0\times T^*M_1\to (T^*M_0)^-\times T^*M_1\\
    (q_0, p_0, q_1, p_1)\to (q_0, -p_0, q_1, p_1)
\end{align*}
where $q_i$ are coordinates on $M_i$ and $p_i = dq_i$ are the corresponding cotangent fiber coordinates.
Under this identification, the diagonal Lagrangian $\Delta$ is sent to $N^*\Delta_M$, where $\Delta_M\subset M\times M$ is the diagonal. Note that we do not check if the canonical diagonal grading on $\Delta$ is sent to the canonical conormal grading on $N^*\Delta_M$; however, this grading information will never appear in our arguments.  
\subsection{Symplectic interpretation of efficient resolutions}
\label{subsec:cobordism}

Lagrangian cobordisms provide a geometric interpretation of ``iterated mapping cones'' in the Fukaya category.
\begin{definition}[\cite{arnol1980lagrange}]
    \label{def:lagrangianCobordism}
    Given Lagrangian submanifolds $\{L_i\subset X\}_{i=0}^k$, a Lagrangian cobordism $K:(L_1, \ldots, L_k)\rightsquigarrow  L_0$ is a Lagrangian submanifold $K\subset X\times \CC$ and has the form
    \[K\setminus \pi_\CC^{-1}(V)= \bigcup_{j=1}^k (L_j\times (j+\RR_{>0})).\]
    outside of a compact set $V\subset \CC$.
\end{definition}
Lagrangian cobordisms provide a geometric decomposition of $L_0$ in terms of the $L_i$. They also give a Floer-theoretic decomposition of $L_0$ in terms of the $L_i$.  
\begin{thm}[\cite{biran2013lagrangian,bosshard2022lagrangian}] \label{thm:birancornea}
    Given an exact\footnote{In \cite{biran2013lagrangian}, the cobordisms are only required to be monotone. In \cite{bosshard2022lagrangian}, the definition of Lagrangian cobordism needs to be slightly modified, as exact Lagrangians for any Liouville form on $X\times \CC$ will fail to satisfy \cref{def:lagrangianCobordism}. See \cite[Definition 3.2]{bosshard2022lagrangian} for the precise definition.} embedded Lagrangian cobordism $K:(L_1, \ldots, L_k)\rightsquigarrow L_0$, there exists an iterated mapping cone decomposition in the (pre-triangulated closure of the) wrapped Fukaya category 
    \[L_0 \cong [L_k \to \cdots \to L_1].\]
\end{thm}

In \cite[Propsition C.2]{hanlon2022aspects}, it was observed that Lagrangian cobordisms pass through geometric composition of Lagrangian correspondences in the following sense: given a Lagrangian cobordism of correspondences $K:(L_1, \ldots, L_k)\rightsquigarrow L_0$ (so that $L_i\in X_1^-\times X_2$), and $L'\in X_1$, there exists a Lagrangian cobordism $K\circ L': (L_1\circ L', \ldots, L_k\circ L')\rightsquigarrow L_0\circ L'$.

Unfortunately, the resulting Lagrangian correspondence may not be exact. If we overlook this detail (which we expect can be overcome by appropriately including bounding cochains), we obtain the following symplectic geometric interpretation of \cref{prop:diagRestoRes}. 
First, observe that when $X_1$ and $X_2$ are totally stopped sectors, then the geometric composition of a product cocore $\cocore$ in $X_1^-\times X_2$ and admissible Lagrangian $L$ of $X_1$ is a complex of cocores.
    
Now suppose that we have a Lagrangian cobordism $K:(\cocore_1, \ldots, \cocore_k)\rightsquigarrow \Gamma_\Delta$, where $\Gamma_\Delta\subset X^-_1\times X_1$ is the graph of the diagonal, and the $\cocore_i$ are the union of product cocores. Then we obtain a (possibly non-exact) Lagrangian cobordism 
\[ (\cocore_1\circ L , \ldots, \cocore_k \circ L )\rightsquigarrow \Gamma_\Delta\circ L = L\]
Modulo the exactness hypothesis, we then apply \cref{thm:birancornea} to prove that the collection $\{\cocore_i\}$ generate $L$ in time $k$.
    
For our applications to Liouville domains, we do not need to employ geometric composition; it suffices to use the fact that $\Gamma_\Delta$ is the diagonal bimodule. 

\subsubsection{Lagrangian cobordisms in Liouville domains}
The main tool that we will employ to bound generation time of a Lagrangian is to construct a Lagrangian cobordism that resolves a Lagrangian submanifold into cocores.
The following proposition\footnote{A version of this proposition was given in \cite[Appendix B]{hanlon2022aspects}. However, the proof contains a minor error corrected here.}  generalizes \cite[Proposition 1.37]{ganatra2018sectorial}, which required a ``thinness" condition.
\begin{prop}
    \label{prop:lagrangianCobordism}
    Let $L\subset X$ be an exact Lagrangian submanifold with primitive $df=\lambda_X|_L$. Suppose that there is a subdomain $Y\subset X$ so that:
    \begin{enumerate}
        \item $X$ is the completion of $Y$,
        \item $L\cap Y= L_1\cup \cdots \cup L_k$, where the $L_i$ are not necessarily connected Lagrangians in $Y$ with Legendrian boundary. 
        \item There exist constants $c_1< \ldots < c_k$ so that $f(\partial L_i)= c_i$
    \end{enumerate}
    Then there is an exact Lagrangian cobordism 
    \[K: (L_1, \ldots, L_k)\rightsquigarrow L.\]
    As a consequence, there exists an iterated mapping cone decomposition in the pre-triangulated closure of $\mathcal W(X))$
    \[L\simeq [L_k\to \cdots\to L_1].\]
\end{prop}
\begin{remark}
The fact that the $L_i$ are allowed to be disconnected, as long as $f(\partial L_i)$ is a fixed constant independent of the component, will be crucial for improving the upper bound on Rouquier dimension by the number of critical points to the number of critical values. 
We also observe that the disjoint union $L\coprod K$ is isomorphic to the direct sum of objects $L \oplus K$ in the wrapped Fukaya category. 
\end{remark}

\begin{proof}
    Let $U$ be a small neighborhood of $\partial Y$, and let $V$ be a small neighborhood of $\partial X$.
    By applying a small Hamiltonian isotopy, we can arrange for $f|_{U\cap L_i}=c_i$. 
    We first show that we may reduce to the case where $\lambda_X|_{L_i}=0$. Consider the function $f_i:= f|_{L_i}-c_i: L_i\to \RR$. This function satisfies the property that $f_i|_{(L_i\cap U)}=0$. In particular, we can find functions $\tilde f_i \colon  Y\to \RR$ supported near the $L_i$ which satisfy the property that \begin{align*}
        \tilde f_i|_{L_i}=f_i && \tilde f_i|_{L_j}=0 \text{ if $i\neq j$} && \tilde f_i |_{U}=0
    \end{align*}
    We then consider the modified Liouville form 
    \[\lambda_X' = \lambda_X -\sum_i d\tilde f_i.\]
    which agrees with $\lambda_X$ on $Y\setminus X$, and therefore is Liouville isotopic to $\lambda_X$.
    For this choice of Liouville structure, the Lagrangians $L_i$ are \emph{strictly exact} in the sense that 
    \[ \lambda_X' |_{L_i}=0.\] 
    Observe that $f':=f-\sum_i f_i : L \to \RR$ is a primitive for $c_i$, and that we still have $f'|_{\partial L_i}=c_i$.
    
    We now assume that $\lambda|_{L_i}=0$. For $t\geq 1$, let $\phi_t \colon  \bar X\to \bar X $ be the time-reparameterized Liouville flow satisfying
    \[\phi_t^* \lambda_X = t \cdot \lambda_X\]
    i.e., in the example of the cotangent bundle, this is simply scalar multiplication by $t$ on the fibers.
    Since the $L_i$ are parallel to the Liouville vector field, for all $t\in [1, \infty)$  $\phi_t(L_i)\cap Y = L\cap Y$.
    So there exists $t_1$ with the property that for $t>t_1$ we have 
    \[\phi_{t> t_1 }(L)\cap X = \bigcup_{i=1}^k \bar L_i\cap X.\]
    where $\bar L_i\subset \bar X$ is the cylindrical completion of the $L_i$ in the completion of $X$.
    Let $t_0>1$ be a constant chosen so that $\phi_{t}(L)\cap V= L\cap V$ for all $t\in [1, t_0]$. In particular, for all $(q, t)$ such that $t<t_0$ and $i_t(q)\in V$, we have $f(q)=0$.
    Now consider a smooth non-decreasing $\rho \colon  \RR\to [1, \infty) $ which satisfies:
    \begin{align*}
         \rho(t)=1 \text{ for all $t<0$} &&  \rho(t)=t\text{ for all $t>t_0$}
    \end{align*}
    Consider the Lagrangian suspension cobordism $K^{pre}$ parameterized by 
    \begin{align*}
        j: \bar L\times \RR\to& \bar X\times \CC\\
        (q, t) \mapsto& (i_{\rho(t)}(q), t+\sqrt{-1}\rho'(s)f(q))
    \end{align*}
    The Liouville form we choose for $\CC$ is $s dt$, where $z=t - \sqrt{-1} s$. Then 
    \[j^*\lambda_{X\times \CC}= i_{\rho}^* \lambda_X - t(\rho'df+\rho'' f dt)\]
    has primitive $F(q,t):=f(q)(\rho(t)-t\cdot \rho'(t))$. We now show that $F|_{V\times \CC}=0$. We break into two cases:
    \begin{description}
        \item[Case 1: $t<t_0$.] When $t<t_0$,  $i_{\rho(t)}(q)\in (V\times \CC)$ implies that $i_{\rho(t)}(q)\in V$ and in particular $f(q)=0$, so $F(q, t)=0$.
        \item[Case 2: $t\geq t_0$.] When $t>t_0$, observe that $\rho(t)=t$, so $f(q)(\rho(t)-t\cdot \rho'(t))=0$.
    \end{description}
    Thus, $K^{pre}$ has Legendrian boundary in $\partial(X\times \CC)$, and its restriction gives a conical exact Lagrangian submanifold of $X\times \CC$. We now prove that this Lagrangian is a Lagrangian cobordism. Observe that we have 
    \begin{align*}
        K|_{t<t_0}=& L\times \RR_{>t_1}\subset X\times \CC\\
        K|_{t>t_1}=&\bigcup_{j=1}^k (L_j\times (c_j+\RR_{>0}))
    \end{align*}
    so $K$ is a Lagrangian cobordism with the desired ends. 
    
\end{proof}

\subsubsection{Perturbing the Skeleton}
The number of distinct values of the primitive $f: L\to \RR$ on $\partial L_i$ control the length of the resolution for $L$ by $\{L_i\}$. By taking a limit over Liouville subdomains limiting to the skeleton, this becomes the value of  $f$ on the intersection.
In the setting where $L$ is the graph of $df$ in $X=T^*M$, $f(\LL_X\cap L)$ gives the critical values of $f$, motivating the following definition.
\begin{definition}\label{def: lagrangian primitive action}
    Let $X$ be a Weinstein manifold with Liouville form $\lambda$ and skeleton $\LL_X$. Given $L$ a Lagrangian submanifold which intersects that skeleton at its Lagrangian stratum,
    the action values are \[\actval(L,f):=\{f(x) \st x\in L\cap \LL_X\}\] where $f: L\to \RR$ is a primitive for the Liouville form restricted to $L$.
\end{definition}
Since $|\actval(L, f)|$ is independent of the choice of $f$ when $L$ is connected, we will usually write $|\actval(L)|$ for this quantity. 

We recall the definition of a generalized cocore from \cite{ganatra2018sectorial}. Given $x\in \LL_X$, a generalized cocore $\cocore(x)$ is an exact cylindrical Lagrangian intersecting $\LL_X$ exactly once transversely at $x$. Observe that $\cocore(x)$ need not be strictly exact (as opposed to a cocore). 
In a general Liouville manifold, there is no reason for a point $x \in \LL_X$ to admit a generalized cocore.
In our previous example of $T^*M$, every point $x\in\LL_X= M$ has a generalized cocore which is $T^*_xM$. 
When $(X, \omega, Z)$ is a Weinstein domain, there exists a generalized cocore passing through every point belonging to the Lagrangian strata of the skeleton. 
To see this, note that the cotangent fiber through $x$ in the handle $H_x$ containing $x$ is a generalized cocore passing through $x$ if no further handles are attached after $H_x$. 
If additional handles are attached after $H_x$, the cotangent fiber can be perturbed to avoid their attaching Legendrians by a general position argument.

\begin{prop}
    Let $X$ be a Weinstein manifold and let $L$ be a generalized cocore. Then there exists a subdomain $Y \subset X$ completing to $X$ and a conical exact Lagrangian $\cocore(y)$ in $X$ with the following properties:
    \begin{itemize}
        \item $\cocore(y)$ and $L$ are Hamiltonian isotopic via a compactly supported Hamiltonian $H$;
        \item $H|_{\LL_X}=0$; and
        \item $\cocore(y)|_Y$ is a (not generalized) cocore at a point $y \in \LL_X$.
    \end{itemize}
\end{prop}
\begin{proof}
After modifying the Liouville form in a neighborhood of the skeleton, we can take $Y\subset X$ a small subdomain around the skeleton satisfying the first three properties of \cref{prop:lagrangianCobordism}, and additionally so that 
\begin{itemize}
    \item locally around $x\in L\cap \LL_X$ it is modelled on a small Weinstein neighborhood $B^*_\eps U$ where $U\subset \LL_X$ and 
    \item  Furthermore, at each point $x\in L\cap \LL_X$ we can identify $U$ with a small ball in $\RR^n$ with $x$ at the origin, so that $L$ restricted to this neighborhood is parameterized by $i:x\mapsto (x, df).$
\end{itemize}
Because the intersection is transverse, we can choose (in appropriate coordinates) $f=x_1^2+\cdots x_k^2 - x_{k+1}^2- \cdots x_{n}^2$, so that $L$ is modeled on a linear Lagrangian subspace of $T^*(-\eps,\eps)^n$ with the standard symplectic form and primitive, as drawn in \cref{fig:cocoreProfile}
Let $\gamma$ be the blue curve drawn in \cref{fig:cocoreProfile}, which is chosen so that its primitive agrees with the red curve at the boundary and at the skeleton.
    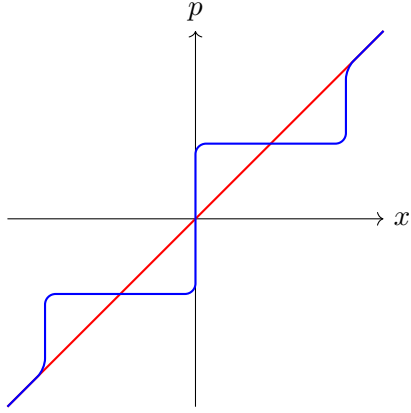
\begin{figure}
        \centering
        \begin{tikzpicture}

\draw[->] (-2.5,0) -- (2.5,0) node[right] {$x$};
\draw[->] (0,-2.5) -- (0,2.5)node[above] {$p$};
\draw[red, thick] (-2.5,-2.5) -- (2.5,2.5);
\draw[rounded corners,blue,thick] (-2.5,-2.5) -- (-2,-2) -- (-2,-1) -- (-1,-1) -- (0,-1) -- (0,1) -- (1,1) -- (2,1) -- (2,2) -- (2.5,2.5);
\end{tikzpicture}         \caption{Replacing a transverse intersection (between the red Lagrangian and the core represented by the $x$-axis) with one restricting to a cocore (blue Lagrangian) in the polydisk $D(r/2)$.}
        \label{fig:cocoreProfile}
    \end{figure}
    Let $\bar \gamma(t)=\gamma(-t)$. Then 
    \[\cocore(y)=\underbrace{\gamma\times \cdots \gamma}_{k} \times \underbrace{\bar \gamma\times \cdots \bar \gamma}_{n-k}\subset D(r,\cdots r).\]
    The Lagrangian $\cocore(y)$ satisfies the following properties:
    \begin{itemize}
        \item agrees with $L$ outside of $D(r/2, \ldots, r/2)$,
        \item is Hamiltonian isotopic to $L$,
        \item admits a primitive $g: \cocore(y)\to \RR$ for the Liouville form which agrees with the primitive $f: L\to \RR$ outside of $D(r/2, \ldots, r/2)$.
        \item Furthermore, in a smaller Weinstein neighborhood $U=B^*_{1/8}D(r, \cdots r)$, we have $g|_{\cocore(y)\cap U}=f(0)$ and $U\cap \cocore(y)$ agrees with a cotangent fiber.   
    \end{itemize}
\end{proof}
    
In particular, for any $L$ intersecting the skeleton transversely at the Lagrangian stratum, there exists $L'$ which is Hamiltonian isotopic to $L$ and $Y\subset X$ so that $L'\cap Y$  is a collection of cocores parallel to the Liouville form and $\actval(L)=\actval(L')$.

\begin{corollary} \label{cor:Weinsteinresolve}
    Let $X$ be a Weinstein domain. Let $L\subset (X, \LL)$ be a Lagrangian submanifold which  meets $\LL_X$ transversely at its Lagrangian strata. Let $\mathcal G$ be the full subcategory of Lagrangian cocores. Then $\Res_{\mathcal G}(L)\leq |\actval(L)|-1$. 
\end{corollary}
In \cref{subsubsec:degenerate}, we discuss generalizations of this result to certain degenerate intersections between $L$ and $\LL_X$.

\subsection{Sectorial LS-covering}
\label{subsec:sectorialLS}
Next, we discuss a geometric version of the LS category for arbitrary sectors. 
We will apply this to cotangent bundles in \cref{subsec:lsUpperBound} and to Lefschetz fibrations in \cref{subsec:sectorialLSLefschetz}.
\begin{definition}
We say that a sectorial covering $\Sigma$ of a Weinstein manifold $X$ consisting of inclusions of sectors $\phi_0, \cdots, \phi_k: X_0, \cdots, X_k \hookrightarrow X$ is a sectorial Lusternik-Schnirelmann covering of $X$ if for any choice of coefficient ring $R$ and $\alpha\in \{0, \ldots, k\}$ we have  $\Rdim^R(\phi_\alpha) = 0$; we let $D(\Sigma)$ denote the depth of the cover. We call it an $R$-sectorial LS-covering if $\Rdim^R(\phi_\alpha) = 0$ for fixed $R$.
We define the \textit{sectorial LS category} of $X$ to be
\[
\LS_{sect}(X) :=  \min_{\text{sectorial-LS covering } \Sigma }
D(\Sigma).
\]
Then $R$-sectorial LS category is defined analogously.
\end{definition}
The depth of such a cover is at most $k$.  
Note that there are no conditions about the intersection of elements in the open cover. In particular, if $X$ admits an arboreal cover (of Dynkin type) using $k+1$ pieces, then $\Rdim(W(X)) \le k$; again, there is no assumption that there intersection of the elements in the cover are themselves arboreal.

While this is the definition we use for this paper, we pose the following open-ended question:
\begin{question}
    Is there a geometric condition on a sectorial inclusion $\phi_i: X_i\hookrightarrow X$ which identifies that $\Rdim^R(\phi_i)=0$ for all rings $R$? One easy condition is that $\phi_i$ factors through the inclusion of $\tilde X_i$ with $\mathcal W(\tilde X_i)=\Perf(R)$ for any ring $R$; for example, if the $\tilde X_i= T^*D^n$. However, this does not handle all of the cases that we would like. 
\end{question}

\begin{lemma} \label{eq:lssect}
   Let $X$ be a Weinstein manifold. For any commutative ring $R$, 
   \[ \Rdim^R(\mathcal W(X)) \le \LS_{sect}^R(X)\le \LS_{sect}(X). \]
\end{lemma}
\begin{proof}
    By \cite{ganatra2018sectorial}, we have $\mathcal W(X)=\hocolim_\Sigma \mathcal W(X_I)$, where $X_I=\bigcap_{i\in I} X_i$. 
    Observe that $\phi_I: X_I\to X$ factors through $\phi_i$ for $i \in I$, and thus $\Rdim^R(\phi_I)=0$ for all $I$. Therefore, the localization $\Grot(\mathcal W(X_I))\to \mathcal W(X)$ is an example of $R$-Lusternik-Schnirelmann covering of $\mathcal W(X)$, from which the bound follows by \cref{prop:TLSCatToRdim}.
\end{proof} 
\section{Applications: Bounds on categorical dimensions from geometry and topology}
\label{sec:applications}
In this section, we apply the tools from \cref{sec:symplecticTools} to provide upper bounds for the diagonal dimension, Rouquier dimension, and triangulated LS category of $\mathcal W(X)$, handling the special case of $X=T^*M$ separately. We summarize the bounds with a comparison to previous approaches here. 
\begin{itemize}
    \item \cite{bai2021rouquier} shows that for any Weinstein domain $X$ with skeleton $\LL_X$ and Hamiltonian diffeomorphism $\phi$, the number of non-degenerate intersection points $|\LL_X \cap \phi(\LL_X)|$ gives an upper bound on generation time of the diagonal bimodule by product Lagrangians. 
    Here, we use a Lagrangian cobordism argument to improve this result to a bound on the diagonal dimension by the number of distinct \textit{action values} of these intersection points, as well as certain degenerate critical points. We first focus on the case of cotangent bundles in \cref{subsec:ddimLeqCritM}, where this can be re-expressed in terms of the minimal number of critical values of a Morse function $f: M\to \RR$. In \cref{subsec:ddimLeqActM}, we handle the general case. Note that our bounds hold for any choice of coefficient ring $R$ and choice of grading/orientation data. 
    \item \cite{bai2021rouquier} also shows that given a sectorial decomposition of $X=\bigcup_i X_i$ with arboreal intersections,
    (more generally, a covering into pieces with $\Rdim(\mathcal W(X_I))=0$) the Rouquier dimension is upper bounded by the depth of the cover. We show that one can weaken this condition and consider covers whose essential images have Rouquier dimension zero to get an improved bound by the triangulated LS category. In \cref{subsec:lsUpperBound}, we apply this to $T^*M$ and show that the bound is at most $\LS(M)$. Additionally, we apply this method to bound the Rouquier dimension of a Lefschetz fibration by the number of critical values in \cref{subsec:sectorialLSLefschetz}. A bound by the number of critical points was obtained in \cite{bai2021rouquier}.
\end{itemize}
Finally, we provide an argument showing that in the case of $T^*M$ with its canonical grading/orientation data, the cuplength of $M$ provides a lower bound for the generation time by a cotangent fiber, which was shown independently in \cite{favero2023rouquier}.
\begin{rem}
    Note that \cite[Example 1.1]{nadler2016wrapped} shows that an alternative symplectic model for $\mathcal{C}(M)$ is the category of wrapped microlocal sheaves. The main result of \cite{ganatra2018microlocal} more directly relates wrapped microlocal sheaves and $\mathcal{W}(T^*M)$. In the category of wrapped microlocal sheaves, a cotangent fiber corresponds to a co-representative of the stalk functor. In fact, \cite{nadler2009microlocal} gives a resolution of the diagonal in the (unwrapped) microlocal setting from which it seems plausible to deduce \eqref{eq:bccotangent}. 
\end{rem}

\subsection{\texorpdfstring{$\Ddim(T^*M)$}{Diagonal dimension of cotangent bundles}: bounds by critical values}
\label{subsec:ddimLeqCritM}
\begin{corollary}\label{cor:non_degenerate_critical_values}
Let $M$ be a smooth manifold, and $N\subset M$ be an embedded submanifold. Let $f: N \rightarrow \mathbb{R}$ be a Morse function, i.e., a smooth function with non-degenerate critical points. Then for any commutative ring $R$, 
\begin{equation}
\gentime_{T^*_xX}N^*N \le |\critval(f)|-1
\end{equation}
\end{corollary}
\begin{proof}
    Let $\tilde f: M\to \RR$ be a smooth function such that $\tilde f|_N=f$ and $\tilde f$ grows quadratically in the normal direction to $N$. Then, let $L$ be the image of $N^*N$ under the time-1 Hamiltonian flow on $T^*M$ given by $\tilde f$. Observe that $L$ intersects $M$ transversely at the critical points of $f$ and that the action of each $q\in L\cap M$ is given by $f(q)$.
    By \cref{cor:Weinsteinresolve}, $L$ is equivalent to a twisted complex of cocores indexed by the critical values of $f$. 
\end{proof}
\begin{corollary} \label{cor:Morsebound}
    Let $f: M\to \RR$ be a Morse function. Consider the diagonal Lagrangian  submanifold $\Delta\subset (T^*M)^-\times T^*M$. Let $\mathcal G$ be the full subcategory generated on the single object $T^*_qM\times T^*_qM$. Then 
    \[\Res_\mathcal G  \Delta  \leq |\critval(f)|-1\]
As a result, we have the following bound on $\Rdim$ of $\mathcal W(T^*M; R)$ valid over a commutative ring $R$
\[
\Rdim(\mathcal W(T^*M; R)) \le
\min_{f:M\to \RR \text{ Morse}} |\critval(f)|-1+ \gentime_{ R}(\Perf R).
\]
    \label{cor:genTimeCocore}    
\end{corollary}
\begin{proof}
    By \cref{eq:cotangentCorrespondences}, the conormal to the diagonal in  $T^*M\times T^*M= T^*(M\times M)$ is identified with $\Delta \subset (T^*M)^- \times T^*M$ by a symplectic involution. The statement on Rouquier dimension follows from \cref{prop:diagGenerationToGeneration}.
\end{proof}

 \cref{subsubsec:degenerate} contains a generalization of \cref{cor:Morsebound} to functions with a certain class of degenerate critical points.

\begin{remark}
    In \cite{abouzaid2012wrapped}, Abouzaid gives a prescription for presenting a compact exact Lagrangian submanifold $L\subset T^*M$ as a twisted complex of cotangent fibers at points in $M\cap L$, which agrees with ours when $N=M$. Our result extends this to non-compact exact Lagrangian submanifolds and observes that the differential is filtered by the action.
\end{remark}

\begin{remark}
The generation time of the zero-section and the diagonal are closely related, at least in the case when $M$ has trivial tangent bundle. 
In this case, $\Delta: M \rightarrow M \times M$ has a thickening to a codimension zero inclusion $\Delta': M \times D^n \rightarrow M \times M$
and an induced proper inclusion 
\[
T^*M \times T^*D^n \hookrightarrow T^*(M\times M) \cong T^*M \times T^*M.
\]
Furthermore, the image of the stabilized zero-section 
$M \times T^*_0 D^n \subset T^*M\times T^*D^n$ under this proper inclusion is the Lagrangian conormal $N_{\Delta_{M}} \subset T^*(M\times M)$, which corresponds to the diagonal $\Delta_{T^*M} \subset (T^*M)^{-} \times T^*M$ under the symplectomorphism between $T^*M$ and $T^*M^{-}$. Furthermore, this proper inclusion maps cotangent fibers $T^*_q M \subset T^*M$ map to product cotangent fibers $T^*_q M \times T^*_q M$. Hence if $M \subset T^*M$ is generated in time $k$ via cotangent fibers, then so is the diagonal $\Delta_{T^*M} \subset (T^*M)^{-}\times T^*M$. Conversely, if we work with field coefficients and the diagonal is generated in time $k$ by cotangent fibers, then by applying convolution, so is the zero-section. 
\end{remark}

\begin{remark}
A useful set of examples are partially wrapped Fukaya categories whose stop is determined by a stratification of the base.
Let $\Sigma$ denote the data of a stratification $\{M_\alpha\subset M\}_{\alpha\in \mathcal S}$, along with the data of a conical sub-bundles $\{\sigma_\alpha\subset N^*M_\alpha\}_{\alpha\in \mathcal S}$. This defines $\LL_{\Sigma}$, a stratified isotropic subset of $T^*L$.
Suppose that we have  $L\subset X$ a strictly exact Lagrangian submanifold with a choice of Weinstein neighborhood so that $\LL_X|B^*_\epsilon L = \LL_\Sigma$ for some $\Sigma$.
Consider now a Morse function with the property that the graph of  $df$ meets $\LL_\Sigma$ transversely. The action values are defined to be the values of $f$ where $df\in \LL_\Sigma$.  
Then the generation time of $L$ by cocores is at most the number of action values. This was the strategy employed in \cite{hanlon2023resolutions}
\end{remark}

\subsection{\texorpdfstring{$\Ddim(\mathcal W(X))$}{Diagonal dimension of wrapped Fukaya categories}: bounds by action values and Reeb chords}

\label{subsec:ddimLeqActM}

Next, we generalize the previous results to general Weinstein domains that are not cotangent bundles. 
We first note that in the case of cotangent bundles, a smooth function $f: M\rightarrow \mathbb{R}$ gives a Hamiltonian function $H_f := f\circ \pi: T^*M\rightarrow M \rightarrow \mathbb{R}$ and the time-1 flow $\phi: T^*M\rightarrow T^*M$ has the property that $M \cap \phi(M) = M \cap \Gamma(df)$ coincide with the critical points of $f$. More generally, suppose $\lambda_X$ is a Liouville form for a Weinstein domain $X$ 
and consider a Hamiltonian isotopy $\phi_t: X \rightarrow X$ (with $\phi:= \phi_1$) induced by a possibly time-dependent Hamiltonian $H_t: X \rightarrow \RR$; let $v_t$ be the Hamiltonian vector field associated to $H_t$. 
Then the \textit{symplectic action} of the Hamiltonian trajectory $\phi_t(x), 0\le t \le 1,$ is defined by integrating $\lambda(v_t) - H_t$ over  $\phi_t(x)$. We can think of the symplectic action as a function $A_{H_t, \lambda}: X \rightarrow \RR$ defined by 
\begin{equation} \label{eq:sympact}
A_{H_t, \lambda}(x) = \int_0^1 (\lambda(v_t) - H_t)\circ \phi_t(x) dt .
\end{equation}
We first have the following result for arbitrary Weinstein domains, which improves \cite[Proposition 1.14]{bai2021rouquier} by replacing intersection points with distinct action values of intersection points with the skeleton. Here, we again make no assumption on grading/orientation data.
\begin{prop} \label{prop:ddimgeneral}
Let  $X$ be a Weinstein domain with Liouville form $\lambda$ and skeleton $\LL_X$. Let $\phi: X\to X$ be the time-$1$ flow of a Hamiltonian $H_t: X \to \RR$ such that $\LL_X\cap \phi(\LL_X)$ is a transverse intersection of the Lagrangian strata of the skeleton. Then, over any commutative ring $R$,
\begin{equation*}
\Res_{\mathcal G}\Delta_{X}  \le 
|A_{H_t, \lambda}(\LL_X\cap \phi(\LL_X) )| -1
\end{equation*}
where $\mathcal G$ is the category of products of cocores.
\end{prop}
\begin{rem}
In the case when $X$ is $T^*M$ equipped with the canonical Liouville form $\lambda_{can}$ and $H_t$ is time-independent $H$ that depends only on the position coordinate, i.e., $H: T^*M \rightarrow M \rightarrow \RR$, then $\phi_t$ flows in the momentum direction, $\lambda_{can}(v_t) = 0$, and $A_{H, \lambda_{can}}(x) = -H(x)$. 
\end{rem}
\begin{proof}
We will use Corollary \ref{cor:Weinsteinresolve} to produce a resolution of the diagonal Lagrangian $\Delta_X \subset X^- \times X$. First, we note that the Hamiltonian isotopy $\phi_t$ induces a Hamiltonian isotopy $Id\times \phi_t: X^-\times X \rightarrow X^-\times X$ that takes $\Delta_X$ to the graph $\Gamma_\phi$ of $\phi= \phi_1$; in particular, $\Delta_X$ and $\Gamma_\phi$ are equivalent in the wrapped Fukaya category of $X^-\times X$ and it suffices to produce a resolution of $\Gamma_\phi$. 
The non-degeneracy condition implies that $\Gamma_\phi \subset X^-\times X$ intersects $\LL_X^- \times \LL_X$ in isolated points $(p, \phi(p))$, where $p \in \LL_X$ and $\phi(p) \in \LL_X$. Note that $\LL_X^-\times \LL_X$ is the skeleton of $X^-\times X$ associated to the Liouville form  $-\pi_1^*\lambda + \pi^*_2 \lambda$, and the points $(p, \phi(p))$ belong to the Lagrangian stratum of the skeleton.  Furthermore, the cocores to $(p, \phi(p))$ in $X^- \times X$ correspond to the product of the co-core (in $X$) associated to $p$ and the co-core associated to $\phi(p)$. Hence $\Gamma_\phi$ has a resolution into product cocores. 
To bound the generation time of $\Gamma_\phi$, we note that by \cite[Proposition 9.19]{mcduff1997introduciton}, the restriction of the Liouville form $-\pi_1^*\lambda + \pi^*_2 \lambda$ to the graph Lagrangian $\Gamma_\phi$ has primitive given by the symplectic action function, i.e., if we consider the parametrization $Id \times \phi: X \rightarrow X^-\times X$ of $\Gamma_\phi$, then 
$$
(Id \times \phi)^*(-\pi_1^*\lambda + \pi^*_2 \lambda) = d(A_{H_t, \lambda})
$$
where $A_{H_t, \lambda}: X \rightarrow \RR$ is given by \eqref{eq:sympact}. Hence, the action, in the sense of \cref{def: lagrangian primitive action} and \cref{cor:Weinsteinresolve}, of the intersection point $(x, \phi(x))$ of $\Gamma_\phi$ and the skeleton $\LL_X^- \times \LL_X$  is precisely $A_{H_t, \lambda}(x)$. Using \cref{cor:Weinsteinresolve}, we obtain the claimed result. 
\end{proof}

Unlike in the case of cotangent bundles, in general if $f: X \rightarrow \mathbb{R}$ is a Weinstein Morse function, then the number of critical points or critical values of $f$ \textit{does not} give an upper bound on $\Rdim(\mathcal W(X))$, as observed by \citeauthor{bai2021rouquier} \cite{bai2021rouquier}. For example, by work of the third author \cite{lazarev2020simplifying}, any high-dimensional Weinstein domain $X^{2n}, n \ge 3$, has a Weinstein Morse function with the number of critical points at most the number of critical points of a smooth Morse function, possibly plus one more. On the other hand, there are exotic cotangent bundle of spheres containing $T^*T^n_{std}$; hence their Rouquier dimension is at least $n$ while they all admit Weinstein Morse functions with two critical points. The issue is that a Weinstein Morse function does not give a perturbation of the skeleton with self-intersection corresponding to the critical points of the Morse function. As we show below, there are more intersection points corresponding to the Reeb chords of the attaching spheres. However, if one counts critical points and also accounts for the Reeb chords of the attaching Legendrian spheres, then there is an upper bound on the Rouquier dimension, as we now explain.

Recall that any Weinstein domain $X^{2n}$ has a self-indexing Weinstein Morse function. In particular, $X$ has a Weinstein presentation as $X_{sub} \cup_{\Lambda} H^n$, where $X_{sub}$ is the \textit{subcritical} portion of $X$ consisting of all Weinstein handles of index \textit{less than} $n$ and $\cup_\Lambda H^n$ denotes the \textit{critical} Weinstein handles of index equal to $n$ attached along $\Lambda$, a link of Legendrian spheres contained in the contact boundary $\partial_\infty X_{sub}$ of $X_{sub}$. 
By \cite{Cieliebak_split}, $X_{sub}^{2n}$ is Weinstein homotopic to $X_{0}^{2n-2} \times D^2$ for some Weinstein domain $X_0^{2n-2}$ of real dimension two less than that of $X$; here, $D^2$ is equipped with the standard Weinstein structure 
corresponding to an index $0$ Weinstein handle and has Liouville form $\frac{1}{2}(xdy- y dx)$, or $\frac{1}{2}r^2 d\theta$ in polar coordinates. We note that there is significant freedom in picking $X_0$. For example, if $X_0, X_0'$ are almost symplectomorphic Weinstein domains, then $X_0 \times D^2, X_0' \times D^2$ are Weinstein homotopic; furthermore, if we vary the number of critical Weinstein handles in the presentation of $X$ (say by doing handle creation or cancellation), the homotopy type of $X_0$ may  change drastically  as well. 

In general, the Legendrian link $\Lambda \subset \partial_\infty(X_0 \times D^2)$ may have infinitely many Reeb chords, even after generic Legendrian perturbation. Indeed, the contact form on $X_0 \times S^1$ given by the restriction of the Liouville form $\lambda_{X_0} + \frac{1}{2}r^2 d\theta$ is $\lambda_{X_0} + \frac{1}{2}d\theta$. Hence, the Reeb vector field is $2\partial_\theta$ and so the Reeb flow on $X_0 \times S^1 \subset \partial(X_0\times D^2)$ is rotation in the $S^1$-direction, which is periodic.  However, if we restrict the Reeb dynamics to the region $X_0\times (S^1\backslash \{-1\})$ of $\partial_\infty(X_0\times D^2)$, then we break this periodicity and $\Lambda$ will generically have finitely many Reeb chords, all of which are non-degenerate.
To achieve this, we consider the Weinstein hypersurface $X_0 \times \{-1\} \subset X_0 \times S^1 \subset \partial_\infty(X_0 \times D^2)$, whose complement is 
$X_0 \times (S^1\backslash \{-1\})$. 
We observe that $\Lambda \subset \partial_\infty(X_0 \times D^2)$  can be Legendrian isotoped in  $\partial_\infty(X_0 \times D^2)$ to a Legendrian $\Lambda' \subset \partial_\infty(X_0 \times D^2)$ that is disjoint from the Weinstein hypersurface $X_0 \times \{-1\} \subset \partial_\infty(X_0 \times D^2)$ and hence $\Lambda' \subset X_0 \times (S^1\backslash \{-1\})$. Indeed, the Weinstein hypersurface $X_0$ is a thickening of a singular Legendrian (namely, the skeleton $\LL_{X_0}$ of $X_0$) and hence $\Lambda$ can be perturbed to be disjoint from $X_0$ for dimension reasons; again, we note that $\Lambda' \subset  X_0 \times (S^1\backslash \{-1\})$ is not unique and depends on the choice of perturbation (which affects how $\Lambda'$ becomes linked with $X_0\times \{-1\}$). 
Finally, since the Reeb vector field on $X_0 \times (S^1\backslash \{-1\})$ is rotation in the $S^1$-direction,  after a generic Legendrian isotopy, we can assume that $\Lambda$ has finitely many non-degenerate Reeb chords.  Indeed, the Reeb chords correspond to  critical points of certain functions, namely the difference of height functions defining different branches of $\Lambda$, and by perturbing one of these height functions, the critical points of the difference function are non-degenerate. 
In conclusion, any Weinstein domain $X$ has a (highly non-unique) presentation as $X_0 \times D^2 \cup_{\Lambda} H^n$, where $\Lambda \subset X_0 \times (S^1\backslash \{-1\})$ has finitely many non-degenerate Reeb chords. In the following result, we show that there is a Hamiltonian isotopy of $X$ whose self-intersection points correspond either to the Reeb chords of $\Lambda$ or to the critical Weinstein handles $\cup_\Lambda H^n_\Lambda$ (and the latter all have the same action). 

For the statement of our result, we need to recall that the action $A(\gamma)$ of a Reeb chord $\gamma$ in a contact manifold $(Y, \xi)$ with contact form $\alpha$ is given by 
$$
A(\gamma):= \int_\gamma \alpha.
$$
Hence in the contact manifold $X_0\times (S^1\backslash \{-1\})$ with contact form $\lambda_{X_0} + \frac{1}{2}d\theta$ and Reeb vector field $2\partial_\theta$, 
the action $A(\gamma)$ of a Reeb chord $\gamma$ (which is constant in the $X_0$-direction) is given by half the difference of angles between the endpoint of $\gamma$ and the starting point of $\gamma$. 
We note that the Reeb vector field $\frac{1}{2}\partial_\theta$ on $S^1\backslash \{-1\}$ is not complete but this is not an issue for us since $S^1\backslash \{-1\}$ is an auxiliary contact manifold living inside $S^1$ where this vector field is complete; alternatively, we can replace  $S^1\backslash \{-1\}$ with the contactomorphic contact manifold 
$(\RR, dx)$ with complete Reeb vector field $\partial_x$. 

\begin{prop}\label{prop: attaching_sphere_reeb_chords}
Suppose that $X$ is a Weinstein manifold with a presentation of the form 
\begin{equation} \label{eq:Weinsteinpres}
    X_{0} \times D^2\cup_{\Lambda} H^n
\end{equation}
where $\Lambda  = \coprod_{j \in J} \Lambda_j\subset X_0 \times (S^1 \setminus \{-1 \}) \subset \partial_\infty (X_0 \times D^2)$ is the link of Legendrian attaching spheres of the critical Weinstein handles with  finitely many non-degenerate Reeb chords $\{ \gamma_i \}_{i \in I}$ in  $X_0 \times (S^1 \backslash \{-1\})$ (by the preceding discussion, any Weinstein domain admits such a presentation). 

Then, there exists a Hamiltonian diffeomorphism $\phi: X \rightarrow X$ generated by  a time-independent $H: X \to \RR$ so that there is a decomposition of intersection points $\LL_X \cap \phi(\LL_X))=\{x_{\gamma_i}\}_{i \in I} \coprod \{x_{\Lambda_j}\}_{j \in J}$
and there is an increasing function $f$ so that  $A_{H, \lambda}(x_{\gamma_i}) = f(A(\gamma_i))$ and $A_{H, \lambda}(x_{\Lambda_j})$ is an arbitrarily small constant that is independent of $j$ and less than $A_{H, \lambda}(x_{\gamma_i})$ for all $i \in I$. 
In particular, there is an equality 
$$
|A_{H, \lambda}(\LL_X \cap \phi(\LL_X))| = \left|\{A(\gamma_i)\}_{i \in I}\right|+1
$$
and hence
\[
\Ddim(\mathcal W(X; R)) \le \left|\{A(\gamma_i)\}_{i \in I}\right|.
\]
\end{prop}

\begin{proof}
We note that the final inequality on diagonal dimension immediately follows from \cref{prop:ddimgeneral} given that we can construct $H$ with the claimed properties. Therefore, we are left to provide the geometric construction of $H$.

We begin by describing the skeleton of $X_0\times D^2 \cup_\Lambda H^n$ in this  Weinstein presentation. The Liouville form on $X_0\times D^2$ is given by $\lambda_{X_0} + \frac{1}{2}r^2 d\theta$ with Liouville vector field $v_{X_0} + \frac{1}{2}r\partial_r$, where $v_{X_0}$ is the Liouville vector field on $X_0$ and $r$ is the radial coordinate on $D^2$. Hence, the skeleton of $X_0 \times D^2$ is given by $\LL_{X_0} \times \{0\}$, where $\LL_{X_0}$ is the skeleton of $X_0$ with respect to $v_{X_0}$ and $\{0\}$ is the origin of $D^2$, i.e., the skeleton of the radial Liouville vector field on $D^2$. 
Then, the relative skeleton $\LL_{(X_0 \times D^2, \Lambda)}$ of the stopped Liouville manifold $(X_0 \times D^2, \Lambda)$ is obtained from $\LL_{X_0} \times \{0\}$ by adding a cone over $\Lambda$ using the negative Liouville flow, i.e.,
$\LL_{(X_0 \times D^2, \Lambda)} = \LL_{X_0} \times \{0\} \cup_{t \le 0} \Phi_{t}(\Lambda)$ where $\Phi_{t}$ is the time $t$ Liouville flow.
In particular, the projection of this relative skeleton to $D^2$ consists of a union of radial lines that start at the origin and end on $\partial D^2$. In addition, note that the relative skeleton $\LL_{(X_0 \times D^2, \Lambda)}$ has boundary given by $\Lambda$ and the skeleton of the domain $X_0 \times D^2 \cup_\Lambda H^n$ is obtained by adding $n$-cells  to $\LL_{(X_0 \times D^2, \Lambda)}$ along $\Lambda$, i.e., 
$\LL_X = \LL_{(X_0 \times D^2, \Lambda)} \cup_{\Lambda} D^n$. 

Recall that by assumption, $\Lambda$ is disjoint from $X_0 \times \{-1\}$ and hence the projection of the relative skeleton $\LL_{(X_0 \times D^2, \Lambda)}$ to $D^2$ is disjoint from $\{-1\} \in \partial D^2$; that is, the projection consists of a set of rays between the origin and the boundary $\partial D^2$ and these rays never make angle $\pi$ with the positive $x$-axis. In fact, we can assume that these rays make arbitrarily small angle with the $x$-axis. Indeed, there is a family of contact embeddings 
\[ f_t: \left(X_0 \times (S^1\backslash \{-1\}), \lambda_{X_0} + \frac{1}{2}d\theta \right) \hookrightarrow \left(X_0 \times (S^1\backslash \{-1\}),  \lambda_{X_0} + \frac{1}{2}
d\theta \right)  \]
given by  $f_t(x, \theta) = (\Phi_{t,0}(x), e^t\theta)$, where $\Phi_{t,0}$ is the Liouville flow on $X_0$. Applying these embeddings to $\Lambda \subset X\times (S^1\backslash \{-1\})$ for large negative time $t \ll 0$, we get that $f_t(\Lambda) \subset X_0 \times (S^1\backslash \{-1\})$ has small angle. 
Hence, we can assume that the Legendrian attaching spheres $\Lambda\subset \partial(X_0\times D^2)$ project to points in $\partial S^1$ with small angle and hence the same holds for the relative skeleton $\LL_{(X_0 \times D^2, \Lambda)}$; to be concrete, we suppose that this angle is less than $\pi/4$ from the positive $x$-axis. Furthermore, since $\Lambda$ has only finitely many Reeb chords, we can assume that there are no Reeb chords with action less than some small positive constant $2\epsilon$. In particular, by the Weinstein neighborhood theorem, there is a tubular neighborhood of $\Lambda$ strictly contactomorphic to $(T^*_{\le \epsilon} \Lambda \times [-\epsilon,\epsilon]_s, \lambda_{std} + ds)$ (after possibly decreasing $\epsilon$) to which the critical Weinstein handles are attached. Here $T^*_{\le \epsilon} \Lambda$ is the cotangent bundle over $\Lambda$ consisting of covectors of norm less than or equal to $\epsilon$ with respect to the standard metric on $\Lambda$, which is identified with a disjoint union of \textit{standard} spheres when we attach Weinstein handles. 

Next, we will construct a Hamiltonian isotopy of $X$ and describe its effect on the skeleton $\LL_{X}$ and the relative skeleton $\LL_{(X_0 \times D^2, \Lambda)}$. 
First, let $f: [0,1] \rightarrow \RR$ be a smooth function with $f(0) = 0$ with positive first derivative $f'(x)$ satisfying $f'(x) = \pi/4$ for $x \in [0,1/3]$, $f'(x) = \epsilon$ for $x \in [2/3, 1]$, and $f'$ is strictly decreasing in $(1/3, 2/3)$. 
Now, using polar coordinates $(r, \theta)$ on $D^2$, we define a Hamiltonian $H: D^2 \rightarrow \mathbb{R}$ as 
$H(r, \theta) = f(r^2)$. 
As the symplectic form on $D^2$ in polar coordinates is $r dr \wedge d\theta$, the associated Hamiltonian vector field $v_H$ is $2f'(r^2) \partial_\theta$ and the Hamiltonian flow is counter-clockwise rotation with the amount of rotation decreasing as the radius increases; the time-one flow is rotation by $\pi/2$ for radius less than $1/3$ and rotation by $2 \epsilon$ for radius greater than $2/3$. 
By pre-composing with the projection to $D^2$, we obtain a Hamiltonian (which we also call $H$ by abuse of notation) $H: X_0 \times D^2 \rightarrow \mathbb{R}$ defined as $H(x, r, \theta) = f(r^2)$. This Hamiltonian flow is constant in the $X_0$-component and the same counter-clockwise rotation in the $D^2$-component. 

Next, we extend this Hamiltonian to the critical Weinstein handles $\cup_\Lambda H^n$. We can take each Weinstein handle $H^n$ to be $T^*_{\le \epsilon}D^n$ with $\partial D^n$ corresponding to a component of the link $\Lambda$. We can also choose the Liouville vector field on $T^*D^n$ to be induced by the vector field $-\frac{1}{2}R\partial_R$ on the zero section $D^n$, where $R$ is the radial coordinate on $D^n$. We note that if we extend $X_0 \times D^2$ slightly into the handle, then $R$ and $r$ are related by $R = 2-r$ (both coordinates can be defined using the Liouville flow in the same way since  the projections of the Liouville vector fields to $D^2$ or $D^n$ is $\frac{1}{2}r \partial_r$ or $-\frac{1}{2}R \partial_R$ and $R$ increases from $1$ when extending past $\partial D^n$ into $X_0\times D^2$). 
Note that $H$ restricted to $X_0 \times S^1$ is constant (equal to $f(1)$) and hence $H$ is constant on each component of $\Lambda$, which corresponds to $\partial D^n$; likewise, $H$ is constant on the attaching region $T^*_{\leq \epsilon}\Lambda \times [-\epsilon, \epsilon]$, which corresponds to $\partial D^n$ and $T^*_{\le \epsilon} \partial D^n \times [-\epsilon, \epsilon]$. 
Furthermore, $H$ is quadratic in the $r$-variable on $D^2$. 
We extend the Hamiltonian to $T^*D^n$ as follows. First, we pick a function $h:[0,1] \rightarrow \mathbb{R}$ that is quadratic up to a constant with the property that $h$ has a maximum at the origin, $h(1) = f(1)$, and $h'(1) = -f'(1) =  - \epsilon$. Then, we define a function (that we also call) $h: D^n \rightarrow \mathbb{R}$ given by $(R, \Theta) \mapsto h(R)$. 
We define the Hamiltonian $H: T^*D^n \overset{\pi}{\rightarrow} D^n \overset{h}{\rightarrow} \mathbb{R}$; that is, $H$ depends only the radial $R$ component of the position coordinate and equals $h(R)$. 
Since both the Hamiltonian on $X_0 \times D^2$ and in $T^*D^n$ are quadratic functions of their respective radius and $r = 2-R$, we obtain a smooth Hamiltonian on $X$. 
Finally, we observe that the Hamiltonian flow on $T^*D^n$ induced by $H: T^*D^n \rightarrow \mathbb{R}$ is given by fiberwise addition of $dh$ since $H$ factors through the projection to the zero-section.

Next, we analyze the effect of the time one flow $\phi$ of $H$ on $\LL_X$.
In what follows, it may be useful to consult \Cref{fig:weinstein_skeleton_isotopya} for depictions of $\LL_X$ and $\phi(\LL_X)$ and the corresponding intersection points. 
First, note that this Hamiltonian preserves the subdomain $X_0 \times D^2$.
In the region $X_0 \times D^2_{\le 1/3}$ of radius at most $1/3$, the Hamiltonian diffeomorphism is rotation by $\pi/2$; since $\LL_{(X_0 \times D^2, \Lambda)}$ is contained in the region with angles between $-\pi/4$ and $\pi/4$, we have that $\phi(\LL_{(X_0 \times D^2, \Lambda)}) \cap \LL_{(X_0 \times D^2, \Lambda)} \cap X_0 \times D^2_{\le 1/3}$ consists of only $\LL_{X_0} \times \{0\}$. 
In the region with radius at least $2/3$, the Hamiltonian diffeomorphism is rotation by $2\epsilon$. The skeleton $\LL_{(X_0 \times D^2, \Lambda)}$ in this region is the cylinder on $\Lambda$. Recall that the Reeb vector field on $X_0 \times S^1$ is $2\partial_\theta$, and so the time-1 Hamiltonian flow on $X_0 \times S^1$ corresponds to time $\epsilon$-Reeb flow. Since $\Lambda$ has no Reeb chords with action less than $\epsilon$, $\phi(\LL_{(X_0 \times D^2, \Lambda)}) \cap \LL_{(X_0 \times D^2, \Lambda)} \cap X_0 \times D^2_{\ge 2/3}$ is the empty set.

Next, we analyze the region consisting of radii between $1/3$ and $2/3$; our analysis of actions is similar to Section 3c) of \cite{Seidel_biased_view}. 
At radius $r$, the Hamiltonian vector field is $v_H = 2f'(r^2) \partial_\theta$. Since the Reeb vector field is $2\partial_\theta$,  if there exists a Reeb chord $\gamma$ of $\Lambda$ of action $T$ and an $r$ with $f'(r^2) = T$, then $t \mapsto (r, \gamma(Tt))$ is a Hamiltonian trajectory $x_\gamma$ that starts and ends on $\LL_{(X_0\times D^2, \Lambda)}$, i.e., corresponds to an intersection point of $\phi(\LL_{(X_0\times D^2, \Lambda)}) \cap \LL_{(X_0\times D^2, \Lambda)}$. 
All Reeb chords have action less than $\pi/4$ (since $\Lambda$ projects to points with angle between $-\pi/4$ and $\pi/4$ and the Reeb vector field is $2\partial_\theta$). Since $f''< 0$, we have that for each $T$, there is exactly one $r$ with $f'(r^2) = T$ so that Reeb chords correspond to Hamiltonian trajectories. Since the Liouville form $\lambda$ in $X_0 \times D^2$ is $\lambda_{X_0}  + \frac{1}{2}r^2 d\theta$, the action of this Hamiltonian trajectory is given by 
$$
A_{H, \lambda}(x_\gamma) = r^2f'(r^2) - f(r^2)
$$
Then, using that $f'' < 0$, we have
$$
\partial_r(r^2 f'(r^2) - f(r^2)) = 2 r^3f''(r^2) < 0 .
$$
Hence as $r$ increases, the action of the corresponding Hamiltonian chord decreases. Furthermore, as $T$ decreases, $r$ decreases since $f'(r^2) = T$ and $f'' < 0$. In conclusion, as $A(\gamma) = T$ increases, so does $A_{H, \lambda}(x_{\gamma})$ as claimed.

Finally, inside each critical Weinstein handle $T^*D^n$, the skeleton is the zero-section $D^n$ and the Hamiltonian flow is given by adding $dh$ in the cotangent fiber direction. 
Thus, $D^n \cap \phi(D^n)$ consists only of constant Hamiltonian trajectories that correspond to critical points of $h$. 
Since $h$ has only one critical point at the origin $0$ of $D^n$, there is one intersection point for each critical Weinstein handle. Furthermore, at constant Hamiltonian trajectories, the action is given by the value of the Hamiltonian, which is $h(0)$ for all critical Weinstein handles. 

Therefore, the intersections $\phi(\LL_X) \cap \LL_X$ correspond to the Reeb trajectories of $\Lambda$ (in an action-increasing way), the critical Weinstein handles (which all correspond to  constant Hamiltonian trajectories with the same action), and the subcritical part of the skeleton $\LL_{X_0} \times \{0\} \subset X_0\times D^2$. To remove this last intersection $\LL_{X_0} \times \{0\}$, we apply a further Hamiltonian isotopy $\psi$ that is supported in $X_0 \times D^2_{\le 1/3}$, i.e. is identity near the boundary. This isotopy is induced by a Hamiltonian isotopy on $D_{\le 1/3}$ that is identity on the boundary and shifts the origin downward and keeps the sector  with angles $3\pi/4$ to $5\pi/4$ within the left hemisphere, i.e., angles $\pi/2$ to $3\pi/2$; see \Cref{fig:weinstein_skeleton_isotopyb}. 
In particular, applying $\psi$ displaces $\LL_{X_0} \times \{0\}$ without creating new intersection points and preserves the old intersection points, as well as their actions. In conclusion, $\psi\circ \phi(\LL_X) \cap \LL_X$ corresponds just to the Reeb trajectories of $\Lambda$ (in an action-increasing way) and the critical Weinstein handles (which all give constant Hamiltonian trajectories with the same action). 
\end{proof}

\begin{figure}
\centering
\begin{subfigure}[t]{.4\linewidth}
\centering
\includegraphics[width=5cm, trim={0 3,5cm 40cm 0},clip]{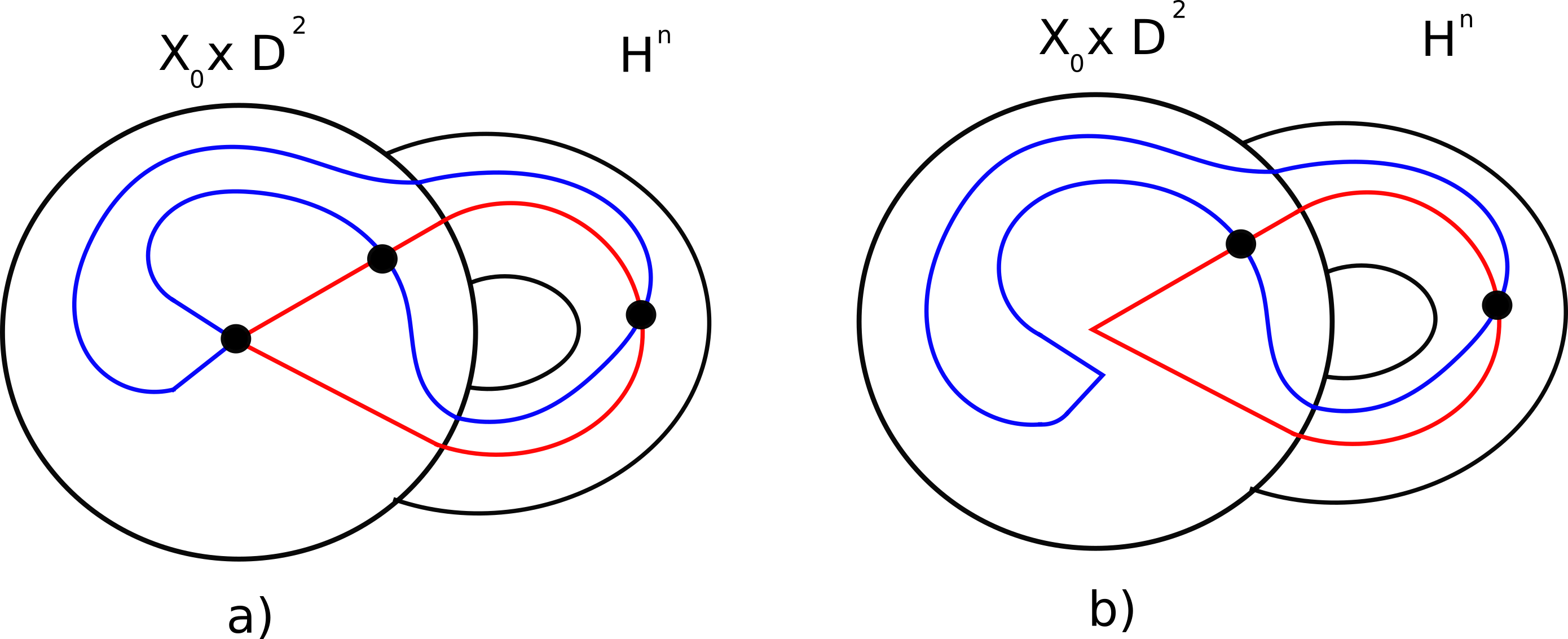}
\caption{The Weinstein skeleton $\LL_X$ in red and its perturbation $\phi(\LL_X)$ in blue; $\LL_X \cap \phi(\LL_X)$ has three components (depicted as black dots) corresponding to (from left to right) 
$\LL_{X_0} \times \{0\}$, the Reeb chords of the attaching Legendrians, and the critical handles.}\label{fig:weinstein_skeleton_isotopya}
\end{subfigure}\;\;\;\;\;
\begin{subfigure}[t]{.4\linewidth}
\centering
\includegraphics[width=5cm, trim={40cm 3.5cm 0 0},clip]{weinstein_skeleton_isotopy.png}
\caption{The Weinstein skeleton $\LL_X$ in red and its perturbation $\psi \circ \phi(\LL_X)$ in blue; now $\LL_X \cap \psi \circ \phi(\LL_X)$ has two components corresponding to the Reeb chords of the attaching Legendrians and the critical handles.}\label{fig:weinstein_skeleton_isotopyb}
\end{subfigure}

\caption{} 
\end{figure}

\begin{remark}
The attaching isotropic spheres of the subcritical handles generically have no Reeb chords between them. Hence we can consider handles of all indices, weighted by the number of Reeb chords (with weight zero for the subcritical handles); this weighted sum will give an upper bound on $\Ddim(\mathcal W(X))$.
\end{remark}

\begin{example}
For example, $T^*S^n$, for $n \ge 2,$ has a Weinstein presentation of the form $D^2 \times D^{2n-2} \cup H^n_{\Lambda_{unknot}}$, where $\Lambda_{unknot} \subset \mathbb{R}^{2n-1} \subset 
\partial(D^2 \times D^2)$ is the Legendrian unknotted sphere.  This Legendrian has a single Reeb chord in $\mathbb{R}^{2n-1}$. This is compatible with the fact that $T^*S^n$ has a Hamiltonian diffeomorphism $\phi$ so that $|S^n \cap \phi(S^n)| = 2$ and  hence
$$
\Rdim(\mathcal W(T^*S^n)) \le  1.
$$
\end{example}
 \subsection{Sectorial LS covers of \texorpdfstring{$T^*M$}{cotangent bundles} from LS-covers}
\label{subsec:lsUpperBound}
In this section, we show that $\LS(M)$ bounds $\LS_{cat}^R(\mathcal W(T^*M; R))$ which also gives a bound on Rouquier dimension by \cref{prop:TLSCatToRdim}. 

We need a preparatory lemma.
\begin{lemma}
    Let $\{U_\alpha\}_{\alpha=1}^k$ be an LS covering of compact manifold $M$. Then there exists $\{i_\alpha:V_\alpha\to M\}_{\alpha=1}^k$ and an LS covering of $M$ where each $V_I=\bigcap_{\alpha\in I} V_\alpha$ is a closed manifold with corners.
    \label{lem:LSManifold}
\end{lemma}
\begin{proof}
    For each $U_\alpha$, consider the function $d_\alpha:U_\alpha\to \RR$ which measures the distance from the boundary. For $t\in \RR$, let $U_\alpha^t:= d_\alpha^{-1}(t, \infty)$. 

    For each $x\in \partial U_\alpha$, there exists an $\eps>0$ so that $B_\eps(x)$ is contained in at least one other $U_j$. Since $M$ is compact, we can choose $\eps$ sufficiently small so that for all $\alpha$, and $x\in \partial U_\alpha$, there exists $\beta\neq \alpha$ so that $B_\eps(x)\subset U_\beta$. Therefore, the open subsets $\{U_\alpha^t\}$ form a cover for $M$ whenever  $t\in [0, \eps/2)$. They also form an $LS$ cover as $U_\alpha^t\subset U_\alpha$.

    While the $d_\alpha$ are not smooth, they can be approximated by smooth $f_\alpha:M\to \RR$. We choose the approximation in such a way that the level sets of $d_\alpha, f_\alpha$ are close i.e. for all $t\in \RR,$
    \[d(d_\alpha^{-1}(t), f_\alpha^{-1}(t))< \eps/16.\]
    Simply put, the level sets, and hence superlevel sets, are close to each other.
    Choose a regular value of $t\in(3\eps/16, \eps/4)$. Then $V_\alpha^{pre}:=f_\alpha^{-1}(t, \infty)$ is an open manifold with boundary and satisfies $U_\alpha^{\eps/2} \subset V_\alpha^{pre}\subset U_\alpha$.

    Finally, we need to make the boundaries transverse. Let $H_\alpha=\partial V_\alpha^{pre}$. Take $H_\alpha'$ a small perturbation of $H_\alpha$ so that 
    \begin{itemize}
        \item The $H_\alpha'$ intersect regularly
        \item  $d(H_\alpha', H_\alpha)<\eps/16$,
        \item $H_\alpha$ are smoothly isotopic to $H_\alpha'$
    \end{itemize}
    Then let $V_\alpha$ be the connected component of $U_\alpha$ which is bounded by the hypersurface $H_\alpha'$. It satisfies the property that $U_\alpha^{\eps/2}\subset V_\alpha \subset U_\alpha$, and therefore forms an LS-cover.
\end{proof}

\begin{prop}
    \label{prop:LSboundLS}
Let $M$ be a smooth compact manifold, possibly with boundary. Then for any commutative ring $R$ and the canonical grading/orientation data on $T^*M$,
 $$
 \LS_{cat}^R(\mathcal{W}(T^*M; R)) \le \LS(M) 
 $$
\end{prop}
\begin{remark}
In particular, in the case when $R$ is a field, we obtain $\Rdim(\mathcal W(T^*M; R))\le \LS(M)$.
\end{remark}
\begin{proof}
Suppose that there is a covering $\{i_\alpha \colon  U_\alpha \hookrightarrow M\}_{\alpha=0}^k$ of $M$ such that $i_\alpha$ is nullhomotopic for all $\alpha$. By \cref{lem:LSManifold}, we may assume that the $U_\alpha$ are manifolds with regularly intersecting boundaries. Then the $T^*U_\alpha$ form a sectorial cover of $T^*M$. For $I\subset \{0, \ldots, k\}$, let $U_I=M \cap \bigcap_{\alpha\in I} U_i$, and for $J\subset I$ let  $i_{IJ}: U_I\to U_J$. We define $\phi_{IJ}:\mathcal W (T^*U_I)\to \mathcal W(T^*U_J)$ be the induced map on cotangent bundles; similarly let $\psi_{IJ}: \Perf(C_{-*}(\Omega U_I))\to \Perf(C_{-*}(\Omega U_J))$ be the induced map on perfect modules over chains on the based loop space. 

We now show that $\Rdim^R(\phi_{I\emptyset}: T^*U_I\to T^*M)=0$. First, we observe that there exist choices of quasi-equivalences $\mathcal F_I : \Perf \mathcal W(U_I)\to \Perf(C_{-*}(\Omega U_I))$ so that the following diagram commutes for all $J\subset I$ (where  $U_\emptyset = M)$
\[
\begin{tikzcd}
\Perf \mathcal W(T^*U_I) \arrow{r}{F_I} \arrow{d}{\phi_{IJ}} & \Perf(C_{-*}(\Omega U_I))\arrow{d}{\psi_{IJ}}\\
\Perf \mathcal W(T^*U_J) \arrow{r}{F_J}  & \Perf(C_{-*}(\Omega U_J))\\
\end{tikzcd}
\]
It is therefore enough to prove that $\Rdim^R(\psi_{I\emptyset}: \Perf(C_{-*}(\Omega U_I))\to  \Perf(C_{-*}(\Omega M)))=0$. Observe that the map $i_{I\emptyset}$ is also null-homotopic, so there is a homotopy equivalence between $\psi_{I\emptyset}$ and the map  $\psi_{Ix}(\Perf(C_{-*}(\Omega U_I))\to  \Perf(C_{-*}(\Omega M))$ which sends $U_I\mapsto x$, where $x\in M$ is some fixed point. Further, $\Rdim^R(\psi_{Ix}:\Perf(C_{-*}(\Omega U_I))\to \Perf( C_{-*}(\Omega M))=0$. 

Since we now have $\Rdim^R(\phi_{I\emptyset}: \mathcal W(T^*U_I)\to \mathcal W(T^*M))=0$, and $\mathcal W(T^*M)$ can be presented as a homotopy colimit over such maps, we obtain the desired bound by \cref{prop:TLSCatToRdim}.
\end{proof}

\begin{remark}
    \label{rem:diagonalLS}
    In fact, a similar argument can be used to bound $\Ddim(\mathcal W(T^*M;R))$ by $\LS(M)(1+\gentime_R(\Perf(R)))$. Given $U_i$ a cover of $M$, one can obtain a cover of $\Delta\subset M\times M$ by sets $V_\alpha=B_\eps(U_\alpha)$, small thickenings of the $U_\alpha$. The $V_\alpha$ are nullhomotopic in $M\times M$. Let $V=B_\eps(\Delta)$ be a small neighborhood of the diagonal. By the same argument as above, the essential image of $\mathcal W(T^*V;R)$ in $\mathcal W(T^*(M\times M);R)$ has Rouquier dimension bounded by the product of 
    \begin{itemize}
        \item the depth of the covering of $V$ by the $V_\alpha$, which is $\LS(M)$
        \item one plus the maximum of the Rouquier dimensions of $j_\alpha: \mathcal W(T^*V_\alpha; R)\to \mathcal W(T^*(M\times M);R)$, which is $\gentime_R(\Perf(R))$
    \end{itemize}
    giving us a bound on the generation time of the diagonal.
\end{remark}

It would be desirable to give a purely symplectic-geometric proof of \cref{prop:LSboundLS}. A first step towards this goal is the following observation
\begin{prop}
Let $f_t: U\to M$ be a nullhomotopy, where $f_0:U\to M$ is an embedding of an open submanifold. Then there exists $U'\subset M$ an open ball so that for any Lagrangian submanifold $L\subset T^*U$ there exists a $L'\subset T^*U'$ which is Lagrangian cobordant to $L$.
\end{prop}
\begin{proof}
First, observe that given two manifolds $N, M$ and a smooth map $\phi: N\to M$, we can define a Lagrangian correspondence $L_\phi\subset T^*N\times T^*M^-$ given by the conormal to the graph $\Gamma_\phi\subset N\times M$. Now observe that the graph of  $\phi_t: U\times I \to M$ defines a Lagrangian submanifold $K_{01}\subset (T^*U\times T^*M^-)\times T^*I$; a cobordism of correspondences $L_{\phi_0}\rightsquigarrow L_{\phi_1}$.  
By \cite[Proposition C.2]{hanlon2022aspects}, we observe that $L=L_{\phi_0}\circ L$ and $L'=L_{\phi_1}\circ L$ are Lagrangian cobordant.
\end{proof}
Unfortunately, the Lagrangian submanifold $L'$ may be immersed. It is an expected feature of Lagrangian correspondences that the geometric composition of unobstructed Lagrangian submanifolds is unobstructed (in the sense of admitting a bounding cochain). Besides the technical challenges of defining the Fukaya category in this setting ($T^*M$, with unobstructed Lagrangian submanifolds) \citeauthor{biran2013lagrangian}'s result on Lagrangian cobordism yielding equivalences should generalize to show that $L$, $L'$ are isomorphic objects of the Fukaya category. Since $L'\subset T^*U'$, the cotangent bundle of a ball, it is isomorphic to some $P\tensor T^*_x M$, where $P\in \Perf R$. This strategy would provide a purely symplectic geometric proof of \cref{prop:LSboundLS}.

Another approach to prove \cref{prop:LSboundLS} is to place stronger conditions on $M$. If $M$ is simply connected and spin, and we work with $\ZZ$ coefficients, then \cite[Theorem 1.8]{lazarev2023prime} proves 
that if $L$ is nullhomotopic in $T^*M$ then it is isomorphic to $P\tensor T^*_xM$. This again provides a proof of \cref{prop:LSboundLS}.

Yet a different approach would be to simply cover $M$ by balls. \cite[Theorem 6.1]{singhof1979minimal} shows that whenever $\LS(M)>\frac{\dim(M)}{2}+2$ that there exists a cover of $M$ by exactly $\LS(M)$ balls. A slight weakening of this is to cover $M$ by sets that are contractible (and not simply null-homotopic). This is closely related to the strong LS category, which is bounded by $\LS(M)+1$ \cite{cornea95lscategory}.

\subsection{Sectorial LS coverings of Lefschetz fibrations}
\label{subsec:sectorialLSLefschetz}
A generalized abstract Weinstein Lefschetz fibration \cite[Definition 1.9]{giroux2017existence} is a tuple $(F, \{V_i\}_{i=1}^k)$ where $F$ is a Weinstein domain (the fiber) and $V_i\subset F$ are a collection of exact parameterized Lagrangian spheres (the vanishing cycles). 
We say `generalized' as we slightly extend their definition by allowing each $V_i=\bigsqcup_{j\in I_i} V_{i}^j$ to be a collection of disjoint Lagrangian spheres $V_{i}^j$ in $F$. Observe that the $V_i^j$ with different $i$ may intersect. 
From this data, one can construct a Weinstein domain 
$$
X := (F \times D^2) \cup \left(\coprod_{i,j} H^n_{V_i^j \times \theta_i}\right)
$$
where we attach Weinstein handles $H^n_{V_i^j\times \theta_i}$ to the Legendrian sphere 
$V_i^j \times \theta_i \subset F\times S^1 = \partial(F \times D^2)$. 
$X$ is called the total space of the Lefschetz fibration.
The corresponding abstract Weinstein Lefschetz fibration resembles a Lefschetz fibration with $k$ critical values and $\sum_{i=1}^k |I_{i}|$ critical points.  
By work of Giroux-Pardon \cite[Theorem 1.10]{giroux2017existence}, any Weinstein domain 
admits an abstract Lefschetz fibration. 
In addition, \cite[Section 6.2]{giroux2017existence} discusses the relationship between abstract Weinstein Lefschetz fibrations and other defintions of Lefschetz fibrations.

\citeauthor{bai2021rouquier} \cite{bai2021rouquier} showed that if $X$ is the total space of an abstract Lefschetz fibration $\pi: X \rightarrow \mathbb{C}$, then
\[ \Rdim(\mathcal W(X)) \le \sum_{i=1}^k|I_i|-1 .\] 
In this section, we improve this result by showing that 
the diagonal dimension as well as the sectorial covering number of an abstract Lefschetz fibration (both of which bound $\Rdim(\mathcal W(X))$), are bounded by $k-1$. 
\begin{prop} \label{prop:Lefschetzddim}
Suppose $(F, \{V_i\}_{i=1}^k)$ is a generalized abstract Weinstein Lefschetz fibration. Then, there is a Hamiltonian $H$ on the total space $X$ whose time one flow  $\phi: X \rightarrow X$ satisfies
$$
|A_{H,\lambda}( \LL_X \cap \phi(\LL_X))| \leq k.
$$
In particular, 
$$
\Ddim( \mathcal W(X)) \le k -1 .
$$
\end{prop}
\begin{proof}
Observe that the Reeb flow on $F\times \mathbb{R}$ is $\partial_z$ in the $\mathbb{R}$-direction and Reeb chords correspond to intersection points of the Lagrangian projection to $F$ of the various $V_i^j\times \theta_i$. Hence, since $V_i^j$ are each embedded in $F$, there are no Reeb chords between the $V_i^j$ for fixed $i$ and varying $j$. There may be Reeb chords between $V_i^j \times \theta_i$  and 
$V_{i'}^{j'} \times \theta_{i'}$
for different $i, i'$ (if the Lagrangian spheres $V_i^j, V_{i'}^{j'}$ intersect). However, all of these Reeb chords have an action which is given by the difference $2\pi i'/k-2\pi i/k$ for some $i' > i$ and hence there are at most $k-1$ different action values. 
Hence, by \cref{prop: attaching_sphere_reeb_chords}, there exists a Hamiltonian diffeomorphism $\phi$ induced by a Hamiltonian $H$ so that $|A_{H, \lambda}(\LL_X \cap \phi(\LL_X))|$ is at most $(k-1) + 1  = k$, and therefore, we obtain an upper bound on $\Ddim(\mathcal W(X))$ by $k-1$ as desired. 
\end{proof}
\begin{example}
    Let $F$ be the $n+2$ punctured cylinder, and let $V_1^1, \ldots, V_1^n$ be a collection of $n$-exact Lagrangian spheres so that $F\setminus V_1$ is a disjoint union of pair of pants. Then the abstract Lefschetz fibration $(F, \{V_i\})$ has Rouquier dimension zero, even though the associated Lefschetz fibration has multiple critical points. 
    
    This is related to the following example (which doesn't fit into the framework of Weinstein domains). 
    Let $X_n$ be the resolution of the $A_n$-singularity by iterated blowup. 
    This is a toric Calabi-Yau. An equivariant section of the toric canonical bundle vanishes on the complement of the algebraic torus; because the toric canonical is trivial this provides a map $\pi: X_n\to \CC$ which is a Lefschetz fibration with a single critical value. It follows that the Rouquier dimension of a (hypothetical) Fukaya-Seidel category $\mathcal {FS}(X_n,\pi)$ is zero even though the Lefschetz fibration has multiple critical points.
\end{example}

\begin{remark} 
The Weinstein handle attachment in Lefschetz fibrations is quite special in that the attaching Legendrian sphere in $F \times S^1 = \partial(F\times D^2)$ has embedded Lagrangian projection in $F$. 
On the other hand, if we take an arbitrary Legendrian sphere $\Lambda \subset F\times S^1$ and attach a handle to obtain the domain $F\times D^2 \cup H^n_\Lambda$, the Rouquier dimension $\Rdim(F \times D^2 \cup H^n_\Lambda)$ may be arbitrarily large (depending on the dimension of $F$) and bounded in terms of the number of Lagrangian immersion points in $F$, as discussed in \cref{prop: attaching_sphere_reeb_chords}.
\end{remark}

We can also bound the sectorial covering number of $X$ in terms of the critical values of a Lefschetz fibration.

\begin{prop}\label{prop: Lefschetz_covering_bound}
Suppose $(F, \{V_i\}_{i=1}^k)$ is a generalized abstract Weinstein Lefschetz fibration.  Then, there is a covering of the total space $X$ by sectors $U_1, \cdots, U_k$ so that $\Rdim(\psi_i) = 0$ where $\psi_i: U_i \to X$ are the inclusions. In particular,  
$$
\LS_{sect}(X) \le k-1.
$$
\end{prop}
\begin{proof}
We start with the Weinstein presentation for $X$ 
$$
(F \times D^2) \cup \left(\coprod_{i,j} H^n_{V_i^j \times \theta_i}\right).
$$
We observe that $X$ can be covered by subsectors
$$
X_i := (F \times D^2) \cup \left(\coprod_j H^n_{V_i^j\times \theta_i}\right), F \times (S^1 \backslash Op(\theta_i)))
$$
That is, this sector has ambient domain $F \times D^2 \cup \left(\coprod_j H^n_{V_i^j\times \theta_i}\right)$ and has sectorial boundary
$F \times (S^1 \backslash Op(\theta_i))$, where $Op(\theta_i)$ is a small neighborhood of $\theta_i \in S^1$.
We note that there is a proper 
sectorial inclusion $F_i: X_i \hookrightarrow X$ whose union covers $X$; this is because $X$ is obtained from $X_i$ by attaching more Weinstein handles to the sectorial boundary $F \times (S^1 \backslash Op(\theta_i))$. 
Finally, we observe that $\Perf \mathcal W(X_i) \cong \Perf \mathcal W(\coprod_j T^*D^n_{i,j})$. This is because $\coprod_j T^*D^n_{i,j}$ is a Weinstein subdomain of $X_i$ and $X_i$ is obtained from $\coprod_j T^*D^n_{i,j}$ by attaching subcritical handles. 
Alternatively, we observe that the Weinstein skeleton of $X_i$ is $c_F \cup (\coprod_{j} [0,1] \times V_i^j \cup D_{i,j}^n)$, where $D_{i,j}^n$ is the core disk of the Weinstein handle $H^n_{V_i^j}$; since $c_F$ is subcritical, the result $\Perf \mathcal W(X_i) \cong \Perf \mathcal W(\coprod_j T^*D^n_{i,j})$ also follows. 
Since $\Perf \mathcal W(\coprod_{j} T^*D^n_{i,j}) \cong  \oplus_j \Perf \mathcal W(T^*D^n)$ and the Rouquier dimension of a direct sum of categories is the maximum Rouquier dimension, we have $$
\Rdim\left(\mathcal W\left(\coprod_j T^*D^n_{i,j}\right) \right) = \max_{j}  \Rdim(\mathcal W(T^*D^n_{i,j})) = 0
$$
and therefore, $\Rdim(\mathcal W(X_i)) = 0$. 
Hence, we have a sectorial cover of $X$ by subsectors $X_i$ with $\Rdim(\mathcal W(X_i)) = 0$, proving $LS_{sect}(X) \le k -1$ as desired. 
\end{proof}

\subsection{Lower bound on split-generation time of \texorpdfstring{$T^*M$}{the cotangent bundle} by a fiber}

For a topological space $M$ and commutative ring $R$, $\cuplength(M; R)$ is the largest integer $k$  so that there exist positive degree classes $\alpha_1, \cdots, \alpha_k \in H^*(M; R)$ of positive degree whose cup-product does not vanish: 
$$
\alpha_1 \cup \cdots \cup \alpha_k \ne 0.
$$
For any smooth function $f: M\rightarrow \mathbb{R}$ with isolated critical points and commutative ring $R$, Lusternik-Schnirelmann theory \cite{cornea2003LS} provides the following inequalities
\[
\cuplength(M; R)\le \LS(M) \le |\Crit(f)|.
\] 
In this section, we similarly use the cuplength of $M$ to give a lower bound on generation time of $T^*M$ by cotangent fibers. 
\begin{prop}\label{prop:lower_bound_cuplength}
For any closed smooth manifold $M$ and $\mathcal W(T^*M; R)$ with the canonical grading/orientation data and any commutative $R$  we have 
\begin{eqnarray}
\cuplength H^*(M; R) &\le& 
\gentime_{T^*_q M} M 
\end{eqnarray}
where $T^*_qM, M$ are objects of $\mathcal W(T^*M;R)$.
\end{prop}
\begin{proof}
We will explain how the ghost lemma \cite[ Lemma 4.11]{rouquier2008dimensions} implies that if $M$ is a twisted complex of cotangent fibers, this complex must have at least $\cuplength(H^*(M))$ - 1 many cones. 
Namely, let $a_1, \cdots, a_k \in H^*(M)$ be positive degree cohomology classes so that 
$a_1 \cup \cdots \cup a_k \ne 0$. We can view $a_i$ as endomorphisms of $M\in \mathcal W(T^*M)$ since there is an isomorphism of algebras $H^*(M) \cong WH^*(M, M)$. 
Then the map
$$
 WH^*(T^*_q M, a_i) =0
 $$
since $WH(T^*_q M, M)$ is supported in a single degree while $a_i$ have positive degree. On the other hand, 
\[
WH^*(M, a_1 \cdots a_k) \ne 0 .
\]
So $M, T^*_qM$ and the sequence of morphisms $\{a_i: M\to M\}$ satisfy the conditions of the ghost lemma. Therefore any presentation of $M$ as a twisted complex of cotangent fibers must have at least $k-1$ cones, from which the bound follows.
\end{proof}

\begin{remark}
The choice of grading/orientation data in \cref{prop:lower_bound_cuplength} is necessary to realize the zero section as an object of the wrapped Fukaya category. 
For a general choice of grading/orientation data, the zero section may fail to be an object.
For example, when the grading/orientation data on $T^*\CP^2$ comes from the trivial spin structure, the zero-section is not relatively spin as $w_2(\CP^2) \ne 0$ and therefore not an object of the wrapped Fukaya category. 
Furthermore, there is an example due to Seidel (found in \cite[Section 4.6.1]{abouzaid2014symplectic}), stating that the wrapped Fukaya category of $T^*\CP^2$ with this grading/orientation data is 2-torsion, and therefore is trivial if one uses coefficients where $2$ is invertible, such as $\mathbb{Q}$. On the other hand, the cuplength of $\mathbb{C}P^2$ over $\mathbb{Q}$ is still 2. In particular, the Rouquier dimension can depend on the grading/orientation data. 
\end{remark}

\begin{remark}
\cref{prop:lower_bound_cuplength} was proven independently in \cite[Theorem 2.14]{favero2023rouquier}.
\end{remark}

\section{Rouquier dimension via embeddings} \label{sec:embedding}

In \cite{bai2021rouquier}, arborealization is used to show that if $X$ is a polarizable Weinstein manifold, then its Rouquier dimension is at most $\dim(X)-3$. In this section, we improve that upper bound in certain cases using an alternative approach that avoids arborealization. 
A key step is to use the doubling trick \cite[Example 10.7]{ganatra2018sectorial} to reduce computations of the wrapped Fukaya category to that of stopped cotangent bundles. 
The reason that this strategy is effective is the following extension of \eqref{eq:bccotangent}.  Let $M$ be an analytic manifold with a subanalytic stop $\stp$. 
Then we can bound the Rouquier dimension of $\mathcal W(T^*M, \stp)$ (with its canonical grading/orientation data) by essentially the same argument as in \cite{bai2021rouquier}. Namely, \cite[Section 2.1]{ganatra2018microlocal} shows that every subanalytic stop on $T^*M$ can be refined to $\stp'$, a conormal stop associated with a triangulation of $M$. 
Small positive pushoffs of the cotangent fibers generate the partially wrapped Fukaya category stopped at $\stp'$, and the subcategory consisting of these objects is posetal of depth $\dim(M)$ \cite[Section 5.8]{ganatra2018microlocal}. As $\mathcal W(T^*M, \stp)$ is a localization of $\mathcal W(T^*M, \stp')$, we obtain 
\[\Rdim(\mathcal W(T^*M, \stp))\leq \Rdim(\mathcal W( T^*M, \stp'))\leq  \dim(M).\]
We thank Laurent C\^{o}t\'{e} for explaining this argument to us.

\begin{prop}\label{prop: embedding_upper_bound}
Suppose that $X$ is a Weinstein domain with a strict analytic embedding into $T^*M$ for a real analytic manifold $M$ with non-empty boundary such that the symplectic normal bundle admits a Lagrangian distribution $E$ that is a  trivial vector bundle.  Then 
$$
\Rdim(\mathcal W(X; E)) \le \dim(M)+1
$$    
where $\mathcal W(X, E)$ is the wrapped Fukaya category of $X$ with grading/orientation data induced by the Lagrangian distribution $E$.

\end{prop}

\begin{remark}
We observe that the embedding of $X$ may be of high codimension, and we do not require it to be a proper embedding. 
\end{remark}
\begin{proof}
Given such an embedding of $X$ with $\dim(X) = 2n$ and $\dim(M) = k$, there is a codimension zero strict analytic embedding
$$
X \times T^*[0,1]^{k-n} \hookrightarrow T^*M
$$
and hence an analytic hypersurface embedding into the contact manifold $T^*M \times \mathbb{R} \subset S^*(M \times [0,1])$ (as in \cite[Corollary 7.12]{ganatra2018microlocal}). Then the double $D_{X\times T^*[0,1]^{k-n}} = \partial(X \times T^*[0,1]^{k-n})$ is  a Weinstein \textit{domain} that admits a analytic Weinstein hypersurface embedding 
$$
D_{X \times T^*[0,1]^{k-n}} \hookrightarrow S^*(M\times [0,1])
$$
Furthermore, as observed by Ganatra, Pardon, Shende \cite{ganatra2018microlocal}, there is a full and faithful embedding into the partially wrapped Fukaya category
$$
\mathcal W(X, E) \hookrightarrow \mathcal W(T^*(M\times [0,1]), D_{X \times T^*[0,1]^{k-n}} )
$$
 Furthermore, by considering $M
\times (0,1)$ as an open manifold and considering the \textit{subcritical} domain $T^*(M\times (0,1))$, the functor 
$$
\mathcal W(X, E) \hookrightarrow \mathcal W(T^*(M\times (0,1)), D_{X \times T^*[0,1]^{k-n}} )
$$
becomes a quasi-equivalence.  
On the other hand, as discussed at the beginning of this section, the Rouquier dimension of a cotangent bundle of $M$ equipped with a subanalytic stop is at most the dimension of $M$
 and hence
 $$
\Rdim (\mathcal W(T^*(M\times (0,1)), D_{X \times T^*[0,1]^{k-n}})) \le \dim(M\times (0,1)) = k+1
$$
which proves the claim.
\end{proof}

\begin{remark} \label{rem:twistedstab}
The condition in \cref{prop: embedding_upper_bound} that the Lagrangian distribution is a trivial vector bundle arises due to the fact that Ganatra, Pardon, Shende \cite{ganatra2018sectorial} only prove that the Fukaya category is invariant under untwisted stabilization; see \cite[Remark 7.8]{ganatra2018microlocal}. Hence if the Fukaya category were known to be invariant under twisted stabilization, then the triviality condition in \cref{prop: embedding_upper_bound} could be dropped. Alternatively, as noted in \cite[Remark 7.8]{ganatra2018microlocal}, the sheaf-theoretic analogs of the wrapped Fukaya category are known to be invariant under twisted stabilization and hence \cref{prop: embedding_upper_bound} would hold in those settings without any triviality condition. Therefore, by either assuming the twisted stabilization invariance or working with microlocal sheaves the following argument provides an upper bound of $2n+1$ of the Rouquier dimension for any Weinstein domain with a polarization.

Observe that there is a strict inclusion $i: (X, \lambda_X) \hookrightarrow (T^*X, \lambda_{std})$ given by $i(x) = (x, \lambda_X(x))$, i.e., taking the graph of the Liouville form.
The symplectic normal bundle of this embedding has a Lagrangian distribution given by the Lagrangian distribution on $X$ (which may not be a trivial vector bundle). Then applying \cref{prop: embedding_upper_bound}, we obtain an upper bound of $\dim(X) +1 = 2n+1$, under these assumptions. 

Finally, we remark that if $X$ arises from SYZ mirror symmetry without discriminant locus, it carries a Lagrangian distribution, so this argument provides a bound slightly weaker than $\Rdim(D^b\Coh(\check X))\leq 2\dim(\check X)$ proven by Rouquier \cite[Proposition 7.9]{rouquier2008dimensions}.
\end{remark}

\begin{remark}
Eliashberg and Gromov \cite{eliashberg1992embeddings} proved that any Stein (and hence Weinstein) domain $X$ of real dimension $2n$ admits a proper holomorphic embedding into $\mathbb{C}^{3/2 n +1} \cong T^*\mathbb{R}^{3/2 n +1}$, and hence admits a proper strict embedding, after Weinstein homotopy of $X$. However, their result does not make any conclusions about the symplectic normal bundle.
\end{remark}

\begin{cor}
    If $X$ is a Weinstein domain that is formally symplectomorphic to $T^*M$, then $\Rdim(\mathcal W(X; \sigma)) \le \dim(M)+1$, where $\sigma$ is the grading/orientation data on $X$ obtained as the pullback of the canonical grading/orientation data on $T^*M$ by the formal symplectomorphism.
\end{cor}

We conjecture that this bound can be decreased to $n$ by using the fact that $M \times (0,1)$ is an open manifold and hence has no maxima. As an application,  we note that any exotic Weinstein ball diffeomorphic to $B^{2n}$ will have a wrapped Fukaya category with Rouquier dimension at most $n+1$. 

Conversely,  \cite[Lemma 7.6]{ganatra2018microlocal} gives an h-principle that states the necessary condition for admitting a Weinstein hypersurface embedding is also sufficient. More precisely, they proved the following data is equivalent to a Weinstein hypersurface embedding $X\hookrightarrow T^*M$ up to homotopy:
\begin{itemize}
    \item a smooth map $f: X \rightarrow M$
    \item a splitting $f^*TM = B \oplus \mathbb{R}$ for a vector bundle $B$ over $X$
    \item An isomorphism $TX \cong B \otimes_\mathbb{R} \mathbb{C}$ of complex vector bundles
\end{itemize}

In particular, if $X^{2n}$ is a real analytic Weinstein domain that has a Lagrangian distribution $E$ on its tangent bundle that is a trivial vector bundle, we can apply the above criteria (and \cite[Corollary 7.5]{ganatra2018microlocal}) to the constant map $f$ and conclude that  $X^{2n}$ has an analytic Weinstein hypersurface embedding into $T^*\mathbb{R}^n \times \mathbb{R}$. 
As a consequence, we have the following result further supporting the observation in the introduction of \cite{bai2021rouquier} that there are no known examples of Weinstein $X$ with $\Rdim(\mathcal W(X)) > \dim(X)/2$.

\begin{corollary} \label{cor:trivialembed}
If $X$ is a Weinstein manifold with a Lagrangian distribution $E$ that is a trivial vector bundle, then 
$$
\Rdim(\mathcal W(X; E)) \le \frac{\dim(X)}{2}+1
$$
\end{corollary}

A simple example that is not a cotangent bundle where the conditions of \cref{cor:trivialembed} hold is a punctured torus. 
For a larger class, note that if $X$ is a smooth hypersurface in $\CC^n$, then $X$ has stably trivial tangent bundle which is then trivial by a theorem of \citeauthor{suslin1977stably} \cite{suslin1977stably}.
Hence, we can take a half-dimensional totally real and hence Lagrangian trivial distribution in $TX$ and apply \cref{cor:trivialembed}.

\section{Conjectures and questions...} \label{subsec:conjectures}
\subsection{... about degenerate intersections and \texorpdfstring{$\mathcal W(T^*M)$}{wrapped Fukaya category of the cotangent bundle}}
\label{subsubsec:degenerate}
There are generalizations of \cref{cor:Weinsteinresolve} to the case when $L$ does not intersect the skeleton transversely, at the cost of adding $\gentime_{R}(\Perf R)$ to our bound for generation time. 

First, we consider the setting of \cref{prop:lagrangianCobordism}. Suppose that $L \subset X$ has 
restrictions $L_1, \cdots, L_k$ to $Y \subset X$ with the property that there exist disks $D_i$ in the smooth strata of $\LL_X$ so that $L_i \subset T^*D_i$. Then $L$ is generated by $L_i$ by \cref{prop:lagrangianCobordism}
and furthermore since $\Perf \mathcal W(T^*D_i; R) \cong \Perf (R)$, each $L_i$ is generated by $T^*_0 D_i$ in time at most 
$\gentime_{R}(\Perf R)$, since $T_0^* D_i$ corresponds to $R$ under the equivalence $\Perf \mathcal W(T^*D_i; R) \cong \Perf (R)$. In summary, such $L$ satisfy the inequality
$$
\gentime_{\mathcal G}(L) \le k -1 + \gentime_R(\Perf R)
$$
giving a variant of the inequality from \cref{cor:Weinsteinresolve}.

As a concrete example, let $f: M \rightarrow \mathbb{R}$ be a smooth function with isolated critical points that are \textit{collared} in the sense that there are disks $D_i \subset M$ containing just the $i$th critical point so that, near the boundary of $\partial D_i$, $f$ takes collared form $r^2 g(\theta)$ for some function $g: \partial D_i^n\rightarrow \mathbb{R}$. 
Let $|\critval(f)|$ be the number of critical values of $f$. Then 
\begin{equation}
\gentime_{T^*_q M}(M) =\gentime_{T^*_q M}(\Gamma_{df}) \le |\critval(f)| -1 + \gentime_R(\Perf R) \label{eq:degeneratebound}
\end{equation}
The collared condition above is needed so that we can conclude that $\Gamma_{df}|_{T^*D_i}$ is a Lagrangian with Legendrian boundary in $T^*D_i$.

\begin{conjecture}
 The inequality given by \Cref{eq:degeneratebound}   holds for any function $f: M \rightarrow \mathbb{R}$ with  isolated critical points, even without the collared condition.
\end{conjecture}

Similarly, there is a generalization of \cref{cor:genTimeCocore}  in the case when $f: M \rightarrow \mathbb{R}$ has collared isolated critical points that are possibly degenerate,  at the cost of adding $\gentime_{R}(\Perf R)$ to the inequality above.
Namely,   for such functions, we have the inequality
\[\gentime_\mathcal G  \Delta  \leq |\critval(f)|-1 + \gentime_{R}(\Perf R)\]
To see this, we proceed as in the proof of \cref{cor:genTimeCocore} and note that this  function gives a perturbation $\Gamma_{d(f + Q)}$ 
of the diagonal $\Delta_{T^*M} \subset (T^*M)^{op} \times T^*M$, where $Q$ is a quadratic form in the normal direction of the diagonal. The intersections of 
$\Gamma_{d(f + Q)}$ and the skeleton $M \times M$ correspond to critical points of $f$. Near these critical points, $\Gamma_{d(f + Q)}$ is a Lagrangian disk in $(T^*D^n)^{op} \times T^*D^n$, with Legendrian boundary. Although this Lagrangian disk is not a product of Lagrangian disks in  $(T^*D^n)^{op}$ and $T^*D^n$, it is still generated by 
product cotangent fibers $T^*_q D^n \times T^*_q D^n$
in time 
$\gentime_{R\times R}((\Perf R )^{op} \otimes \Perf R) \le \gentime_R(\Perf R)$, which gives us the extra term $\gentime_R(\Perf R)$ in the above inequality. 

It is also natural to ask whether the same can be done for Lefschetz fibrations, whose critical points are non-degenerate. 
\begin{question}
Is there a class of degenerate singularities so that the number of critical values of a holomorphic function $\pi: X\to \mathbb{C}$ with singularities in this class provides an upper bound for $\LS_{cat}(\mathcal W(X;R))$ or $\Ddim(\mathcal W(X;R))$? 
\end{question}

\subsection{... On choices of coefficients for \texorpdfstring{$\mathcal W(T^*M;R)$}{wrapped Fukaya categories of cotangent bundles}}

We note that the $\LS(M)$ bound for the Rouquier dimension is unfortunately very sensitive to the choice of coefficients:
\[\Rdim(\mathcal W(T^*M; R))\leq \LS(M)\cdot (\gentime_R(\Perf(R))+1).\]
However, we know that $\Rdim(\mathcal W(T^*M; R))$ is frequently much smaller from our bound via the diagonal dimension \cref{cor:Morsebound}. This leads us to conjecture:
\begin{conjecture}(\emph{cf.} \cref{q:rdimvsrdimr})
 Let $X$ be Weinstein and $R$ be a commutative ring. The sectorial LS-bound for the Rouquier dimension is not sensitive to the choice of coefficient in the sense that \[\Rdim(\mathcal W(X; R))\leq \LS_{sec}^R(\mathcal W(X))+\gentime_R(\Perf(R)).\]
\end{conjecture}

\subsection{... On cuplength, \texorpdfstring{$\Rdim(\mathcal W(T^*M;R))$}{Rouquier dimension}, and a result of Hofer } \label{subsubsec:cuplength}

We do not know whether the cuplength provides a lower bound on $\Rdim(\mathcal W(T^*M); R)$ since our result is only for the generation time by the cotangent fiber (and not arbitrary split-generators). However, we conjecture that this lower bound does hold. 
\begin{conj} \label{conj:cuplength}
For any smooth compact manifold $M$ and any commutative ring $R$ and the \textit{canonical} grading/orientation data on $\mathcal W(T^*M; R)$, 
\[
\cuplength H^*(M; R) \le
\Rdim^R(\mathcal W(T^*M; R)).
\]
\end{conj}

Cuplength has previously appeared as an obstruction in symplectic geometry in several contexts. In \cite{hofer1985lagrangian}, it was shown that 
if $\phi \colon  T^*M\rightarrow T^*M $ is a Hamiltonian symplectomorphism, then the number of (not necessarily transverse) intersection points between $M$ and $\phi(M)$ is at least the cuplength of $H^*(M; R)$,
$$
\cuplength H^*(M; R) \le |M \cap \phi(M)| $$
for any commutative ring $R$.
In fact, \cite{hofer1988lusternik,floer1989cuplength} proves this result for compact Lagrangians in a compact symplectic manifold satisfying $\pi_2(X;L)=0$ where $R=\ZZ/2\ZZ$.

The bounds in this paper appear to be related, and we expect they can be generalized to obtain an improvement of Hofer's result in $T^*M$. 
Namely, let $\phi: T^*M\rightarrow T^*M$ be any Hamiltonian symplectomorphism so that each intersection point $q_i \in M \cap \phi(M)$ is collared in the sense of \cref{subsubsec:degenerate}. Then whenever $R$ is a field
\begin{equation}
\cuplength H^*(M; R) \le \actval|M \cap \phi(M)| 
\label{eq:cuplengthbounddegenerate}
\end{equation}
To see this,  suppose that $q_1, \cdots, q_k \in M \cap \phi(M)$ are the intersection points. First, assume that the $q_i$  have distinct action values. Then $M$ is equivalent to a twisted complex on $k$ Lagrangian disks $\Gamma_{df_i} \subset T^*D^n_i$, where $f_i: D_i\rightarrow \mathbb{R}$ is a function with collared isolated critical points on a disk $D_i \subset M$ containing $q_i$. Then $\gentime_{T^*_0 D^n_i}(\Gamma_{df_i}) = 0$ since $\gentime_R(\Perf R) = 0$. Hence, $\gentime_{T^*_q M}(M) \le k$. In the case when the intersection points can have repeated action values, only the number of distinct action values matters. 

\begin{conjecture}
    The inequality from \cref{eq:cuplengthbounddegenerate} holds without the collared condition on the intersection of $L$ and $\phi(L)$.
\end{conjecture}
This naturally to the leads to the following more general question.
\begin{question}
    What is the relation between \cite{hofer1988lusternik,floer1989cuplength} bound for Lagrangians $L$ in compact $X$ with $\pi_2(X, L)=0$ and our generation bound \cref{eq:cuplengthbounddegenerate}?
\end{question}

\subsection{... on generation of \texorpdfstring{$\mathcal W(T^*M)$}{the wrapped Fukaya category of the cotangent bundle}}
\begin{question}
Given a submanifold $i: U \subset  M$, possibly disconnected, then when does $N_U \subset T^*M$ generate/split-generate $\mathcal W(T^*M)$?
\end{question}
One of our main observations is that if $U \subset M$ is null-homotopic, then $N_U$ split-generates $\mathcal W(T^*M)$ over a field. On the other end of the spectrum, we observe that if $i$ has a retract $r$, then $N_U$ cannot generate or split-generate $\mathcal W(T^*M)$; we note that this setting is in some sense orthogonal to the setting when $i$ is null-homotopic. 
To see this, we note that there are functors $$
\mathcal W(T^*U) \rightarrow \mathcal W(T^*M) \rightarrow \mathcal W(T^*U)
$$
induced by $i$ and $r$ respectively (for example, induced by the DGA maps 
$C_*(\Omega U) \rightarrow C_*(\Omega M) \rightarrow C_*(\Omega U)$).
These functors take $U \subset T^*U$ to $N_U \subset T^*M$ and back to $U \subset T^*U$ and also take $T^*_x M$ to $T^*_x U$. 
So if $N_U$ generated/split-generated $\mathcal W(T^*M)$, then it would generate/split-generate $T^*_x M$, and hence its image $U$ in $T^*U$ would generate/split-generate $T^*_x U$. However, $U$ does not generate/split-generate $T_x^*U \subset T^*U$ since the former has a finite-dimensional morphism space while the latter has in general infinite-dimensional morphism space if $U$ is simply connected. 

We observe that most of our results give upper bounds on generation by cotangent fibers, and the lower bound is also for generation for cotangent fibers. However, interesting phenomena can occur by focusing on resolving the category. For example, as observed in \cite{lazarev2022symplectic}, for any object $C$ of a triangulated category, the direct sum
$$
C \oplus C \oplus C[1]
$$
\textit{resolves} $C$ via two cones. On the other hand, this object generates $C$ immediately. Hence $T^*_q M\oplus T^*_q M \oplus T_q^*M[1]$ resolves $T^*_q M$ in time $2$, and we do not know whether this is the minimum number.  
\begin{question}
Is it true that two mapping cones are required for $T^*_q M\oplus T^*_q M \oplus T_q^*M[1]$ to generate $T^*_q M$?
Are there examples of generators $G$ for $\mathcal W(T^*M)$ whose generation time (or even split-generation) time is bigger than that of $T_q^*M$?
\end{question}
The version of the question above for split-generation is precisely the computation of the Orlov spectrum of $\mathcal W(T^*M)$.

\subsection{... about sectorial LS-category and action values}
In the smooth setting, we have the bound 
$$
\LS(M) \le \min_{f:M\to \RR \text{ Morse}}|\critval(f)|
$$
by using the stable manifolds of the critical points of the Morse function to cover the manifold. In fact, Morse can be replaced by $f$ with isolated critical values (see, for instance, \cite{rudyak2003lusternik}). In the case of general symplectic manifolds, one potential analog of the number of critical values of a Morse function is the intersection number of skeleta. We do not know whether the analogous inequality holds:
\begin{question}
Does the following inequality hold
\begin{equation}
\LS_{sect}(X) \le \min_{\phi: X\to X, \LL_X} |A_H(\LL_X \cap 
\phi(\LL_X))|
\label{eq:lasteq}
\end{equation}
where $\phi$ is a Hamiltonian isotopy generated by $H$ and $\LL_X$ a skeleton of $X$, possibly after Weinstein homotopy.    
\end{question}
We note that both of these expressions bound $\Rdim(X)$. This is a geometric analog of \cref{question:ddimvsLS}. In particular, is $LS(M)$ a lower bound for the right hand side of \cref{eq:lasteq} for any skeleton for $T^*M$?

\printbibliography

\Addresses

\end{document}